\documentclass[11pt]{amsart}

\usepackage{graphicx}
\usepackage{amsxtra}
\usepackage{amssymb}
\usepackage{amscd}
\usepackage{enumerate}
\usepackage{hyperref}

\newcommand{\Hz}  {{H^o}}
\newcommand{\Had} {{\bar H}}
\newcommand{\Spec}{\operatorname{Spec}}

\newcommand{\R}   {{\mbA}}
\newcommand{\GZ}  {{\mbG}}
\newcommand{\Hom} {\operatorname{Hom}}
\newcommand{\id}  {{\mti\mtd}}
\newcommand{\Lie} {\operatorname{Lie}}
\newcommand{\tPhi}{{\tilde \Phi}}
\newcommand{\tal} {{\tilde \alpha}}

\newcommand{\tDe} {{\tilde \Delta}}
\newcommand{\tom} {{\tilde \omega}}
\newcommand{\Pic} {\operatorname{Pic}}
\newcommand{\SPic}{\operatorname{SPic}}
\newcommand{\XD}  {X_{\tilde \Delta}}
\newcommand{\gr}  {\operatorname{Gr}}
\newcommand{\tW}  {{\tilde W}}

\newcommand{\mcalT}{\boldsymbol{\mathcal T}}
\newcommand{\lGIT}{\backslash\!\backslash}

\newcommand{\Proj}{\operatorname{Proj}}
\newcommand{\bgrl}{{\underline{\lambda}}}
\newcommand{\bgra}{{\underline{\alpha}}}
\newcommand{\bmu}{{\underline{\mu}}}

\newcommand{\ssupp}{\mathop{\mathnormal supp_{\Omega}}\nolimits}

\newcommand{\divi}{\operatorname{div}}
\newcommand{\Ic}{{I^c}}

\newtheorem{lem}{Lemma}[section]
\newtheorem{teo}[lem]{Theorem}
 
\newtheorem{prp}[lem]{Proposition} 
\newtheorem{cor}[lem]{Corollary}

 \theoremstyle{definition}
 
\newtheorem{oss}[lem]{Remark} 

\theoremstyle{remark}

\newcommand{\mA}  {\mathbb A} 
 \newcommand{\mC}{\mathbb C}  
 \newcommand{\mF}{\mathbb F}

\newcommand{\mN}{\mathbb N}  \newcommand{\mP}{\mathbb P} 
\newcommand{\mQ}{\mathbb Q} \newcommand{\mR}{\mathbb R}

\newcommand{\mZ}{\mathbb Z}

\newcommand{\calA}  {\mathcal A} 
\newcommand{\calB}{\mathcal B}

 \newcommand{\calL}{\mathcal L} \newcommand{\calM}{\mathcal M}
 \newcommand{\calO}{\mathcal O} 
  
\newcommand{\calT}{\mathcal T}  \newcommand{\calV}{\mathcal V}

  \newcommand{\gog}{\mathfrak g}
\newcommand{\goh}{\mathfrak h}

\newcommand{\got}{\mathfrak t}

\newcommand{\mbA}  {\mathbf A} 
   
  \newcommand{\mbG}{\mathbf G}
  
 \newcommand{\mbL}{\mathbf L}

 \newcommand{\mbU}{\mathbf U} 
 \newcommand{\mbX}{\mathbf X}

  \newcommand{\mtd}{\mathrm d} 
  
 \newcommand{\mti}{\mathrm i}

\newcommand{\gra}{\alpha} \newcommand{\grb}{\beta}       
  
 \newcommand{\grl}{\lambda}     \newcommand{\grs}{\sigma}
\newcommand{\grf}{\varphi}      \newcommand{\om} {\omega}

\newcommand{\grG}{\Gamma} \newcommand{\grD}{\Delta}
\newcommand{\grL}{\Lambda} 

\newcommand{\mi}  {\imath}
\newcommand{\mj}  {\jmath}
\newcommand{\mk}  {\Bbbk}

\newcommand{\surietta}{\longrightarrow \!\!>}

\newcommand{\lra}     {\longrightarrow}
\newcommand{\isocan}  {\simeq}
\newcommand{\vuoto}   {\varnothing}

\newcommand{\cech}    {\spcheck}

\newcommand{\coinc}   {\equiv}

\newcommand{\defi}    {:=}

\renewcommand{\geq}   {\geqslant}
\renewcommand{\leq}   {\leqslant}
\newcommand{\senza}   {\smallsetminus}
\newcommand{\ristretto}{\bigr|}

            \newcommand{\st}       {\, : \,}
       
         \newcommand{\mand}     {\text{ and }}
        
  \newcommand{\msuchthat}{\text{ such that }}
  \newcommand{\mforall}  {\text{ for all }}

\title{The quotient of a complete symmetric variety} 

\author{Corrado De Concini}\address{Dip. Mat. Castelnuovo, Univ. di
  Roma La Sapienza, Rome, Italy}\email{deconcin@mat.uniroma1.it}

\author{ Senthamarai Kannan}\address{Chennai Mathematical Institute,
  Chennai, India}\email{kannan@cmi.ac.in} 

\author{Andrea Maffei}\address{Dip.  Mat. Castelnuovo, Univ. di Roma
  La Sapienza, Rome, Italy}\email{amaffei@mat.uniroma1.it}

\dedicatory{Dedicated to Ernest Vinberg on the occasion of his 70th
  birthday }

\begin{document}

\maketitle

\section{Introduction} Let $G$ be a semisimple simply connected
algebraic group over an algebraically closed field of characteristic
different from 2.  Given an involution $\grs$ of $G$ with fixed
subgroup $G^\grs$, we fix a subgroup $G^\grs\subset H\subset
N_G(G^\grs)$.

Our goal in this paper is the study of the action of $H$ on certain
completions of $G/H$ with the methods of geometric invariant theory.

The study of such problems starts with the famous paper \cite{KoRa} of
Kostant and Rallis in which the $H$ action on the quotient $\Lie G /
\Lie H$ is studied. This can be considered as an infinitesimal version
of our study. Results similar to those in \cite{KoRa} have been later
obtained by Richardson in \cite{Rich} in the case of the quotient
$G/H$.

In particular Richardson has proved, among other things, that if we
take the image $S_H$ in $G/H$ of an anisotropic maximal torus $S$ in
$G$ and consider the action of the restricted Weyl group $\tilde W$ on
$S_H$ (see below for the definitions), the GIT quotient $H\lGIT G/H$
is isomorphic to $\tilde W\backslash S_H$. Furthermore he shows that
the closed $H$ orbits in $G/H$ are precisely the orbits of elements in
$S_H$.

In this paper we generalize these two results to the case of a
completion $Y$ of $G/H$. In particular we reprove the results of
Richardson mentioned above.

To state our result take a smooth toroidal projective $G$ equivariant
completion $Y$ of $G/H$. In $Y$ consider the closure $Y_S$ of $S_H$.
The $\tilde W$ action on $S_H$ extends to an action on $Y_S$. Fix an
ample line bundle $\mathcal L$ on $Y$. Our first result is that the
GIT quotient $H\lGIT_{\mathcal L} Y$ relative to $\mathcal L$ is
isomorphic to $\tilde W\backslash Y_S$. In particular this quotient
does not depend on the choice of $\mathcal L$. (Theorem
\ref{teo:GIT}).

We then pass to the study of the set $Y^{ss}\subset Y$ of semistable
points in $Y$ with respect to $\mathcal L$. Also in this case we show
that $Y^{ss}$ does not depend on the choice of $\mathcal L$ (Remark
\ref{indipendence}) and we describe rather precisely the intersection
of $Y^{ss}$ with any $G$ orbit. In particular we show that given two
$H$ orbits $O_1\subset \overline O_2$ in $Y^{ss}$ then they both lie in the same
$G$ orbit (Proposition \ref{lem:Gorbiteechiusure}) and that a $H$
orbit $O$ in $Y^{ss}$ is closed if and only if it meets $Y_S$ (Theorem
\ref{teo:closed}) .  These last facts allow us to give a version of
our results in the case of any $G$ stable open subset in $Y$.

The proofs of our results are rather straightforward in characteristic
zero and are based on the careful analysis of sections of line bundles
on $Y$ given in \cite{Bifet} and \cite{Brion}. However to carry out
our proofs in positive characteristic we have to deal with a number of
rather technical results which often do not appear in the literature
and which, in view of this, we have decided to explain here.

\section{Preliminaries}\label{sec:prel}

In this section we introduce notations, recall some simple properties
and describe the spherical weights relative to a given involution.

Let us choose an algebraically closed field $\mk$ whose characteristic
is not equal to $2$.  Usually all algebraic group schemes in this
paper are going to be affine and defined over $\mk$ but, occasionally
we are going to consider group schemes defined over the ring
$\R:=\mZ[1/2]$ and flat over $\Spec \R$. Gothic letters are going to
denote Lie algebras.

Let $G$ be a semisimple and simply connected algebraic group.  Let
$\bar G$ be the adjoint quotient of $G$ and $Z$ the kernel of the
projection of the isogeny $G \lra \bar G$. This is a possibly not
reduced subgroup of $G$ whose associated reduced subvariety is given
by the center of $G$.

Let $\grs$ be an involution of $G$ and let $\Hz = G^\grs$ be the
subgroup of elements fixed by $\grs$. We consider also the inverse
image $\Had$ under the isogeny $G\lra \bar G$ of the subgroup of $\bar
G$ of elements fixed by the involution of $\bar G$ induced by $\grs$.
We recall that $\Hz$ is connected and reductive and that $\Had$ is a
possibly not reduced subgroup of $G$ whose associated reduced subgroup
is the normalizer $N_G(\Hz)$ of $\Hz$. It is known that the connected
component of the identity of $\Had$ with reduced structure is equal to
$\Hz$ (see \cite{DP1}).

Let now $\Hz\subset H \subset \Had$ be a possibly not reduced subgroup
of $G$. The quotient $G/H = \Spec \mk[G]^H$ is called a
\emph{symmetric variety}.

We fix an \emph{anisotropic maximal torus} $S$ of $G$, that is a torus
of $G$ such that $\grs(s)=s^{-1}$ for all $s\in S$, and having maximal
dimension among the tori with this property. The dimension $\ell$ of
$S$ is called the \emph{rank} of the symmetric variety. We choose also
a $\grs$ stable maximal torus $T$ of $G$ containing $S$ 
  and a Borel subgroup containing $T$ with the
property that the intersection $B\cap \grs(B)$ has minimal possible
dimension. Occasionally we will also need to consider \emph{isotropic
  tori}, that is tori contained in $H$.

\subsection{Ring of definition}\label{ssec:Z}
It will be important for us that the classification of involutions is
independent of the characteristic (see \cite{Spr1}). So we can use Kac
classification or Satake classification to construct the involutions.
If we use Kac classification we see that we can assume that $G, \grs$
and hence $\Hz$ are all defined over $\R$ and that there is a
maximally isotropic maximal torus (this means a maximal torus of $G$
containing a maximal torus of $\Hz$) defined over $\R$ and a $\grs$
stable Borel subgroup containing this torus also defined over $\R$. On
the other hand if we use Satake classification, we see that we can
assume that $G, \grs, \Hz$, the maximal torus $T$, the torus $S$ and
the Borel subgroup $B$ are all defined over $\R$.  However,
occasionally, we will need to work with a $\R$ form of $G$ where both
maximally isotropic and maximally anisotropic maximal tori are defined
and split over $\R$.

We start with a flat $\R$ form $\GZ$ of $G$, and with a $\grs$ defined
over $\R$ constructed using Kac diagrams. So $ \Hz$ is defined over
$\R$ and there exists a $\R$ split maximal torus $N$ of $\Hz$ defined
over $\R$ and a $\R$ split maximal torus $M$ of $G$ defined over $\R$
containing $N$. The characters of $M, N$ are defined over $\R$ and the
root decomposition of the Lie algebra of $G$ is also defined over
$\R$. In particular all Borel subgroups containing $M$ are defined
over $\R$.  Let $B_M$ be a $\grs$ stable Borel subgroups of $G$
containing $M$.  Let $\Psi$ and $\Psi^+\subset \Psi$ be the
corresponding sets of roots and positive roots of the Lie algebra of
$G$ with respect to $B_M$. Finally notice that also $\bar H$, hence
$H$, can be assumed to be defined over $\R$.

We want to show that in $G$ there is a maximally anisotropic maximal
torus defined and split over $\R$.  This slightly strength a result in
\cite{DS}.

\begin{lem}\label{lem:STZ}
  There is a torus $S$ in $G$ defined and split over $\R$ such that
  $\grs(s)=s^{-1}$ for all $s\in S$ and $S$ has maximal dimension
  among the tori with this property. Moreover there is a maximal torus
  $T$ of $G$ containing $S$ defined and split over $\R$.  Finally the
  root decomposition of $\gog$ with respect to the action of the torus
  $T$ is defined over $\R$ and there is a Borel $B$ subgroup
  containing $T$ and defined and flat over $\R$ such that the
  dimension of $\grs(B)\cap B$ is the minimal possible. The two Borel
  subgroups $B$ and $B_M$ are conjugated by an element of $\GZ(\R)$.
\end{lem}
\begin{proof}
  For each root $\grb \in \Psi$ denote by $u_\grb(t)$ the
  corresponding one parameter subgroup. This subgroup can be defined
  over $\R$.  We construct (see \cite{Knapp} \S VI.7) the torus $T$ as
  follows . Let $\calB\subset \Psi^+$ be a set of roots maximal among
  the subsets with the following properties
  \begin{enumerate}[\indent i)]
  \item $\grb \in \calB$ implies $\grs(\grb)=\grb$ and
    $\grs(u_{\grb}(t))=u_\grb(-t)$;
  \item $\grb, \grb' \in \calB$ implies $\grb+\grb', \grb-\grb' \notin
    \Psi$.
  \end{enumerate}
  For each $\grb \in \calB$ set
  $$g_\grb=u_{\grb}(1)u_{-\grb}(-1/2), 
  \mand \quad g_\calB = \prod _{\grb \in \calB}g_\grb.$$ Notice that 
  since by ii) the  roots  in $ \calB$  are orthogonal to  each other, 
  the elements $g_{\beta}$ as $\beta$ runs in $\mathcal B$ commute and   
  $g_\calB$ is well defined and lies in $\GZ (\R )$. 
  
  We then set $T=g_\calB M g_\calB^{-1}$ and $S=g_\calB M_\calB
  g_\calB^{-1}$ where $M_\calB$ is the subtorus of $T$ corresponding
  to the coroots in $\calB$. By \cite{Knapp} \S VI.7 $T$ and $S$ have
  all the required properties.

  Our claims about the root decomposition now follows from the
  analogous properties for the torus $M$, and under this hypothesis it
  is clear that each Borel containing $T$ is defined and flat over
  $\R$.  Also notice that $g_\calB B_M g_\calB^{-1}$ is a Borel
  containing $T$ so it must be conjugated to $B$ by an element in
  $N_G(T)$. Now by Lemma 2.7 in \cite{DS} every element of the Weyl
  group has a representative in $\GZ(\R)$ proving the last claim.
\end{proof}

We finish this section with a simple Lemma regarding invariants.
Since we are going to deal with non necessarily reduced algebraic
groups let us recall that if $L$ is a not necessarily reduced
algebraic group and $V$ is a representation of $L$, a $L$ invariant
vector $v \in V$ is a vector whose image under the coaction $V\lra
V\otimes \mk[L]$ is $v \otimes 1$.

\begin{lem}\label{lem:invsuZ}
  Let $\mbL$ be an algebraic group scheme defined and flat over $\R$
  (we do not assume that $\mbL$ is either connected or reduced in
  general) and let $V$ be a finite dimensional representation of
  $\mbL$ defined and flat over $\R$. Assume that
  $V(\mC)^{\mbL(\mC)}\neq 0$.  Then there is an $\mbL$ invariant
  vector defined over $\R$ whose reduction modulo $p$ is different
  from $0$ for all odd primes $p$. In particular
  $V(\mk)^{\mbL(\mk)}\neq 0$.
\end{lem}

\begin{proof} Let $V_\R$ be an $\R$ lattice compatible with the
  action. The action of $\mbL$ on $V$ is given by the coaction map
  $a^\sharp: V_\R \lra \R[\mbL] \otimes_\R V_\R$.  If $B$ is a $\R$
  algebra, set $a^\sharp_B=\id_B\otimes_\R a^\sharp$.  Thus an element
  $v$ in $V(B):=B\otimes_{\R}V_{\R}$ is fixed by $\mbL$ if
  $a^\sharp_B(v)= 1\otimes v$.  Let now $F:V_\R\lra \R[\mbL]
  \otimes_\R V_\R$ be given by $F(v)=a^\sharp (v) - 1\otimes v$ and
  $F_B = \id_B \otimes F$.  $V(B)^\mbL= \ker F_B$. In particular
  notice that when $B$ is a field of characteristic zero we have,
  since $V_\R$ is a free $\R$ module, that $\ker F_B= B\otimes_\R \ker
  F$. In particular $V(\mC)^{\mbL(\mC)}=\mC\otimes_\R \ker F$ is
  defined over $\R$.  Moreover since also $\R[\mbL]\otimes_\R V_\R$
  has no torsion, we have that if $n \in \mZ\senza\{0\}$, $v\in V_\R$
  and $nv\in \ker F$ then $v\in \ker F$. So $\ker F$ is a direct
  summand of $V_\R$.
  
  It follows that $\mk\otimes_\R \ker F$ injects into a non zero
  subspace of $V(\mk)^{\mbL(\mk)}$ proving our claim.
\end{proof}

\subsection{Spherical weights and the restricted root system}\label{ssec:sferici}

We want to describe now the Weyl modules of $G$ which have a non zero
$H$ invariant vector.

If $A$ is a torus we denote with $\grL_A$ its character lattice
$\Hom(A,\mk^*)$.  Given a surjective homomorphism $A\lra B$ between
tori, we are going to consider $\grL_{B}$ as a sublattice of $\grL_A$.

Let $\grL=\grL_T$ and let $r:\grL \lra \grL_S$ be the surjective
homomorphism induced by the inclusion $S\subset T$. Let also $\Phi$ be
the root system of $\gog$ with respect to $T$, $\Phi^+$ (resp.
$\Delta$) be the choice of positive roots (resp. the simple roots)
corresponding to the Borel $B$ and $\grL^+$ be the dominant weights
with respect to $B$.

Every character $\grl$ of $T$ extends uniquely to a one dimensional
character of $B$ and we define $\calL_{\grl}$ as the line bundle $G
\times_B \mk_{-\grl}$ on $G/B$. Every line bundle on $G/B$ is
isomorphic to a line bundle of this form. For $\grl \in \grL^+$ the
\emph{Weyl module} $V_\grl$ is defined as the dual of the space of
sections $\grG(G/B,\calL_\grl)$.  With the choices of the previous
section, all these objects are defined over $\R$. Furthermore it is
well know \cite{Jantzen} that $\grG(G/B,\calL_\grl)$ and hence
$V_{\lambda}$ is flat over $\R$ . Occasionally we will have to
consider also line bundles on a partial flag variety $G/P$ where $P$
is a parabolic subgroup containing $B$. The natural projection $G/B\to
G/P$ induces an inclusion of Pic$(G/P)$ in $\Pic(G/B)=\Lambda$ and
allows us to identify $\Pic(G/P)$ with the sublattice $\grL_P$ of
$\Lambda$ consisting of those characters $\lambda$ of $B$ which extend
to $P$.  For $\lambda\in \Pic(G/P)$ we are going, by abuse of
notation, to denote by $\calL_\grl$ the line bundle $G \times_P
\mk_{-\grl}$ on $G/P$.

We define the monoid of dominant $H$ \emph{spherical weights} as
$$
\Omega^+_H = \big\{ \grl \in \grL^+ \st \grG(G/B,\calL_\grl)^H \neq 0
\big\}
$$
and the lattice of spherical weights $\Omega_H$ as the lattice
generated by $\Omega_H^+$. We set also $\Omega = \Omega_{\Hz}$ and
$\Omega^+=\Omega_{\Hz}^+$.

Recall that, since $H$ has an open orbit in $G/B$, if $\grl \in
\grL^+$ then the space of $H$ invariant sections
$\grG(G/B,\calL_\grl)^H$ is at most one dimensional. A non zero vector
in $\grG(G/B,\calL_\grl)^H$ will be called a \emph{spherical vector}.

Let us now give a description of $\Omega$. In characteristic zero
$\Omega$ has been described by Helgason \cite{Helg} using analytic
methods. An algebraic proof of these results has been given by Vust
\cite{Vust}.  The Theorem of Vust is stated in characteristic zero but
its proof can be used verbatim in any characteristic different from 2
once we replace $V_\grl$ with $V^*_\grl$. Moreover Vust's proof can
also be easily adapted to describe the lattice $\Omega_H$. Let $S_H =
S/S\cap H$ then we have the following Theorem.

\begin{teo}[Vust \cite{Vust} Th\'eor\`eme 3]\label{teo:Vust}
  Let $\grl \in \grL^+$ then $\grl \in \Omega^+_H$ if and only if
  $\grs(\grl)=-\grl$ and $r(\grl)\in \grL_{S_H}$.
\end{teo}

We will need also to study quasi invariants under the action of
$\Had$, so we define a dominant weight $\grl$ to be
\emph{quasi spherical} if the representation $\grG(G/B,\calL_\grl)$
has a line fixed by $\Had$. We denote the monoid of quasi spherical
dominant weights by $\Pi^+$ and we set $\Pi$ equal to the sublattice
spanned by $\Pi^+$ and call it the lattice of quasi spherical weights.

Quasi spherical weights have been described in terms of spherical
weights and exceptional roots by De Concini and Springer in \cite{DS}.

Let $\Phi_0$ (resp. $\grD_0$) be the set of roots (resp. simple roots)
fixed by $\grs$ and let $\Phi_1$ (resp. $\grD_1$) be the complement of
$\Phi_0$ in $\Phi$ (resp. of $\grD_0$ in $\grD$). With our choices of
the Borel subgroup $B$ we have $\grs(\gra) \in \Phi^-$ for all $\gra
\in \Phi_1^+=\Phi_1\cap\Phi^+$ (see \cite{DP1}). Moreover the
involution $\grs$ induces an involution $\bar \grs$ of $\grD_1$ where
$\bar \grs (\gra)$ is the unique simple root such that $\grs(\gra)
+\bar \grs (\gra)$ lies in the span of $\grD_0$.  A simple root $\gra
\in \grD_1$ is said to be \emph{exceptional} if $\bar\grs(\gra)\neq
\gra$ and $\kappa(\grs(\gra),\gra)\neq 0$, $\kappa$ being a non
degenerate bilinear form on $\grL$ invariant under the action of the
Weyl group.  We denote by $\{\om_\gra\}_{\gra \in \grD}$, the
fundamental weights with respect to the simple basis $\grD$.  We have,

\begin{teo}[De Concini and Springer \cite{DS} Lemma 4.6
  and Theorem 4.8]\label{teo:DS}\hfill

  \begin{enumerate}[\indent i)]
  \item For each $\grl \in \Pi^+$ the line fixed by $\Had$ is
    unique.
  \item $\Pi^+$ is generated as a monoid by $\Omega^+$ and the
    fundamental weights $\om_\gra$ corresponding to the exceptional
    roots.
  \end{enumerate}
\end{teo}

The set of spherical weights is related to the \emph{restricted root
  system} as follows.  Let us quickly recall how restricted roots are
defined. If $\gra \in \Phi$ is not fixed by $\grs$ we define the
\emph{restricted root} $\tal$ as $\gra - \grs(\gra)$ and the
restricted root system $\tPhi \subset \grL$ as the set of all
restricted roots.  This is a (not necessarily reduced) root system
(see \cite{Rich}) of rank $\ell$ and the subset $\tPhi^+$ (resp.
$\tDe$) of restricted roots $\tal$ with $\gra$ positive (resp. $\gra$
simple) is a choice of positive roots (resp. a simple basis) for
$\tPhi$. 

We collect in the following lemma some known and easy consequences of
the previous theorems.

\begin{lem}\label{lem:sferici}\hfill

  \begin{enumerate}[\indent i)]
  \item $\Pi\cap \grL^+ = \Pi^+$ and $\Omega_H \cap
    \grL^+=\Omega_H^+$;
  \item In the adjoint case we have $\Omega_{\Had} = \mZ[\tPhi]$;
  \item In the simply connected case we have$$\Omega = \{\grl \in \grL
    \st \grs(\grl)=-\grl \mand
    \tfrac{2\kappa(\grl,\tal)}{\kappa(\tal,\tal)}\in \mZ \mforall \tal
    \in \tPhi\};$$  \item If $\grl \in \grL$ and $n\grl \in \Omega$ for some positive
    natural number $n$ then $\grs(\grl)=-\grl$;
  \item The restriction of $r$ to $\Omega_H$ is injective and
    $r(\Omega_H)=\grL_{S_H}$.
  \end{enumerate}
\end{lem}

In particular, by $iii)$, $\Omega^+$ is the set of dominant weights of
the root system $\tPhi$, so it is a free monoid of rank $\ell$ and a
basis of it is given by fundamental weights $\tom_\tal$ with respect
to $\tDe$. Notice   that if $\gra$ is exceptional also
$\grb=\bar\grs(\gra)$ is exceptional. If this is the case, we shall call  $\tal \in \tDe$
  an exceptional restricted simple root  and we recall that $\tom_\tal
= \omega_\gra + \omega_\beta$.

Finally we apply Lemma \ref{lem:invsuZ} to our situation.

\begin{cor}\label{cor:invV}
  If $\grl \in \Omega^+_H$ then $V_\grl$ has a nonzero vector fixed by
  $H$ and if $\grl \in \Pi^+$ then $V_\grl$ has a line fixed by
  $\Had$. More precisely there is a vector of $V_\grl$ defined over
  $\R$ whose reduction modulo any odd prime is different from $0$ and
  fixed by $H$ (respectively spans a line fixed by $\Had$).
\end{cor}
\begin{proof}
  Let $G, \grs$, $V_{\lambda}$ be defined over $\R$ as explained above. Let $M$,
  $N$, $B_M$ be as in section \ref{ssec:Z}.  In particular any character
  of the group $\Hz$ is a character of $N$ hence it is  defined over $\R$. 
  
  Let now $\grl \in \Omega^+_H$. Since $V_\grl(\mC)$ contains a non zero
  vector fixed by $H$ the claim follows from    Lemma \ref{lem:invsuZ}.

  In general notice that since $\Hz$ is a spherical subgroup (it has
  an open orbit in $G/B$) it acts on two different lines in
  $V_\grl(\mathbb C)$ stabilized by $\Hz$ with different characters.
  In particular any line in $V_\grl(\mC)$ which is stabilized by
  $\Hz(\mC)$ must be defined over $\R$: indeed let $R$ be such a line
  and consider the character $\chi$ of $\Hz$ given by the action of
  $\Hz$ on $R$. Recall that with our choices all characters of $\Hz$
  are defined over $\R$. By applying Lemma \ref{lem:invsuZ} to $V_\grl
  \otimes \chi^{-1}$ we see that the line $R$ must be defined over
  $\R$. In particular the line stabilized by $\Had(\mC)$ in
  $V_\grl(\mC)$ is stabilized by $\Hz$ so it is defined over $\R$ and
  it is $\Had$ stable.
\end{proof}

\subsection{Line bundles on $G/H$}
In this section we want to study some properties of the line bundles
on $G/H$. We begin with a remark on $\Had/H$.

\begin{lem}\label{lem:HeS}
  The coordinate ring of $\Had/H$ is isomorphic to the group algebra 
  $\mk[\Omega_H/\Omega_\Had]$. 
\end{lem}

\begin{proof}
  Let $H\cap S = H\times_G S$ be  the scheme theoretic intersection of $H$
  and $S$. By Proposition 7 in \cite{Vust} we have $H = \Hz \cdot
  (H\cap S)$.  So we have
  $$\Had/H \isocan \Had\cap S / H\cap S \isocan \ker \{S_H \surietta
  S_\Had\}.$$ 
  where the kernel has to be considered scheme
  theoretically.  Now by Lemma  \ref{lem:sferici} v) we have $S_\Had
  \isocan \Spec \mk[\Omega_\Had]$ and $S_H \isocan \Spec
  \mk[\Omega_H]$.

  It follows that, if we denote by $e^\chi$ the function on $S_H$
  corresponding to $\chi \in \Omega_H$, the coordinate ring of the kernel
  is then given by $\mk[S_H]/\langle e^\chi - 1 \st \chi \in
  \Omega_\Had\rangle \isocan \mk[\Omega_H/\Omega_\Had]$, proving the
  claim.
\end{proof}

We denote by $x_H$ the point of $G/H$ corresponding to the coset $eH$
and by $q_H : G/ H \lra G/\Had$ the projection induced by inclusion $H
\subset \Had$.

The line bundles on $G/H$ are parametrized by the set of one
dimensional characters $\grL_H$ of $H$ by associating to a line bundle
$\calL$ the character by which $H$ acts on the fiber of $\calL$ over
$x_H$. 

If $\grl \in \Pi^+$ by Theorem \ref{teo:DS} the line fixed by $\Had$
in $V_\grl^*$ is unique and we can consider the character
$-\chi_H(\grl)$ given by the action of $H$ on this line. The map
$\chi_H :\Pi^+ \lra \grL_H$ extends to a group homomorphism $\chi_H :
\Pi \lra \grL_H$ and by Lemma \ref{lem:sferici} $i)$ the kernel of
this homomorphism is given by $\Omega_H$.  In particular for any $\xi
\in \Pi/\Omega_\Had$ we can consider a line bundle $\calL_\xi$ on
$G/\Had$ whose associated isomorphism class is given by
$\chi_\Had(\xi)$.

\begin{prp}\label{lem:splitta} 
  The vector bundles $(q_H)_*(\calO_{G/H})$ and $ \bigoplus_{\xi\in
    \Omega_H / \Omega_\Had} \calL_\xi$ on $G/\Had$ are $G$
  equivariantly isomorphic.
\end{prp}

\begin{proof}Set $\Xi_H = \Omega_H / \Omega_\Had$.  Notice first that
  by Lemma \ref{lem:HeS} the map $q_H$ is a covering of degree equal
  to the cardinality of $\Xi_H$. So the two vector bundles
  $(q_H)_*(\calO_{G/H})$ and $\bigoplus_{\xi\in \Xi_H} \calL_\xi$ have
  the same rank.
  
  If $\xi\in \Xi_H$ then $q_H^*(\calL_\xi)$ is trivial for all $\xi$
  as a $G$ linearized line bundle. So, by adjunction, we have a $G$
  equivariant monomorphism of sheaves $\calL_\xi \lra (q_{H})_{
    *}(\calO_{G/H})$.  We deduce for any subset $R\subset \Xi_H$ a $G$
  equivariant map $\gamma_R:\bigoplus_{\xi\in R} \calL_\xi\to
  (q_H)_*(\calO_{G/H})$.  Since $\gamma_R$ is equivariant the induced
  map at the level of the total spaces of vector bundles has constant
  rank.
  
  We claim that $\gamma_R$ is of rank $|R|$. If $|R|=1$ this is clear
  by the above considerations. We proceed by induction. Write
  $R=R'\cup\{\xi\}$. $\gamma_R'$ is of rank $|R|-1$. Assume $\gamma_R$
  is not of maximal rank. We clearly get an inclusion $j:\calL_\xi\to
  \oplus_{\xi'\in R'}\calL_{\xi'}$. In particular there exists
  $\xi'\in R'$ such that the composition of $j$ with the projection
  onto $\calL_{\xi'}$ is a non zero $G$ equivariant morphism and thus
  an isomorphism of line bundles. Since $\xi\neq \xi'$ this is a
  contradiction.
  
  If we apply this to $R=\Xi_H$ and use the fact that
  $(q_H)_*(\calO_{G/H})$ and $\bigoplus_{\xi\in \Xi_H} \calL_\xi$
  have the same rank we get that $\gamma_{\Xi_H}$ is a isomorphism as
  desired.
\end{proof}

\section{Completions of symmetric varieties}\label{sec:comp}
An \emph{embedding} of a symmetric variety $G/H$ is a normal connected
$G$ variety $Y$ together with an open $G$ equivariant inclusion
$\mj_Y: G/H\subset Y$. We set $y_0$ equal to the image of $x_H$ under
this embedding and call it the basepoint of $Y$. We are also going to
consider the finite covering $\pi_Y:G/\Hz\lra Y$ of $\mj_Y$ given by
$\pi_Y(g\Hz ) =g\cdot y_0$.  We denote by $Y_0$ the image of $\mj_Y$
and set $\partial Y = Y \senza Y_0$ and $\grD_Y$ equal to the set of
irreducible components of $\partial Y$ of codimension $1$ in $Y$.

A line bundle $\calL$ on $Y$ is said to be \emph{spherical} if
$\pi_Y^*(\calL)$ is isomorphic to the trivial line bundle on $G/\Hz$.
We denote $\SPic(Y)$ the subgroup of the Picard group $\Pic(Y)$ of $Y$
of spherical line bundles. We also say that a line bundle is \emph{strictly
spherical} if restricted to the open orbit $G/H$ it is isomorphic to the
trivial line bundle and we denote by $\SPic_0(Y)$ the subgroup of
$\Pic(Y)$ of classes of strictly spherical line bundles. 

Many of the properties of $Y$ can be deduced from corresponding
properties of the associated toric variety $Y_S$. This is
defined as the closure of the orbit $S\cdot y_0$ in $Y$. Notice that
since $S\cdot y_0$ is isomorphic to $S_H$, $Y_S $ is a toric variety
for the torus $S_H$.  The normalizer $N_\Hz(S)$ of $S$ in $\Hz$ acts on
$Y_S$ and  the action of the centralizer $Z_\Hz(S)$ of $S$ is
trivial. It follows that we have an action of the \emph{restricted
  Weyl group} $\tW = N_\Hz(S)/Z_\Hz(S)$ on $Y_S$.

We are now going to describe an open subvariety $Y_S^+$ of $Y_S$ with
the property the $\tW$ translates of $Y_S^+$ cover $Y_S$. Let
$\grL\cech_S$ be the lattice of one parameter subgroups of $S$. If
$\eta \in \grL\cech_S$ and there exists the limit $\lim_{t\to
  \infty}\eta(t)\cdot y_0$ we denote this limit by $y_\eta$. We say
that $\eta \in \grL\cech_S$ is positive if $\tal(\eta(t))$ is a non
negative power of $t$ for all $\tal \in \tDe$ and let $Y_S^+ $ the
union of the $S$ orbits of the elements $\{y_\eta \st \eta \in
\grL\cech_S$ is positive$\}$. It is then immediate to verify that
$Y_S=\tilde WY^+_S$. Indeed if $y\in Y_S$ there is a $\eta\in
\grL\cech_S$ and a $s\in S$ such that $y=s(\lim_{t\to
  \infty}\eta(t)\cdot y_0)$. Since $\eta$ is $\tW$ conjugate to a
positive one parameter subgroup we deduce $y$ is $\tW$ conjugate to an
element in $Y_S^+ $.

\subsection{The wonderful compactification of a symmetric variety}\label{ssec:X}
The so called wonderful compactification $X$ of the symmetric variety
$G/\Had$ has been introduced in characteristic zero in \cite{DP1} and
in arbitrary characteristic in \cite{DS}. We want to very briefly
recall some of the basic properties of $X$ and introduce some
notations.

Recall that by Lemma \ref{lem:sferici} and Theorem \ref{teo:Vust} a
basis of the character lattice $\grL_{S_\Had} $ is given by the set
$\tDe = \{\tal_1,\dots,\tal_\ell\}$ of simple restricted roots (with an
arbitrarily chosen numbering). Thus we get an action, defined over
$\R$, of $S_\Had$ on the affine space $\mA^\ell$ given by
$s(a_1,\ldots ,a_\ell)=(\tal_1(s)a_1,\ldots ,\tal_\ell(s)a_\ell)$ . The following Theorem can   be taken implicitly as the defnition of the wonderful compactification.

\begin{teo}[Theorem 3.1 in \cite{DP1}, Proposition 3.10, Theorem 3.10 and
  Theorem 3.13 in \cite{DS}]\label{teo:X} 
The wonderful compactification $X$ of  $G/\Had$ is the unique $G/\Had$ embedding such that
  \begin{enumerate}[\indent i)]
  \item $X$ is a smooth projective $G$ variety and the closure of
    every $G$ orbit in $X$ is smooth.
  \item $\partial X$ is a divisor with normal crossing and smooth
    irreducible components.
  \item given a $G$ orbit $\mathcal O\subset X$, $\overline {\mathcal
      O}$ is the transversal intersection of the irreducible divisors
    in $\grD_X$ containing it.
  \item The intersection of any number of divisors in $\grD_X$ is a
    $G$ orbit closure. In particular the intersection of all divisors
    in $\grD_X$ is the unique closed $G$ orbit in $X$.
  \item There exists a scheme $\mbX$ defined and flat over $\R$ whose
    specialization to $\mk$ is isomorphic to $X$. Moreover the point
    $x_0=\mj_X(x_\Had)$ is defined over $\R$. 
  \item Let $\mbG$ be as in Section \ref{ssec:Z}. There is an action
    of $\mbG$ on $\mbX$ that specializes over $\mk$ to the action of
    $G$ on $X$.
  \item We have an isomorphism $X_S^+ \isocan \mA^\ell$ as $S_\Hz$
    toric varieties defined over $\R$.  \end{enumerate}
\end{teo}

Since any projective $G$ variety is isomorphic to a variety $G/P$ with
$P$ a parabolic subgroup containing $B$, we have already remarked that
its Picard group can be identified with a sublattice of $\grL$. Thus
composing with the homomorphism induced by the inclusion of the unique
closed orbit, we get a homomorphism $j:\Pic(X)\to \grL$. One has,

\begin{teo}[Theorem 4.2 and Theorem 4.8 in \cite{DS}]\label{teo:XI}\hfill
 
\begin{enumerate}[\indent i)]
\item The homomorphism $j$ is injective and its image is the
  sublattice $\Pi$ of $\grL$.
\item The map $D\to j(\calO(D))$ is a bijection between $\grD_X$ and 
    $\tDe$.
\end{enumerate}
\end{teo}

Notice that combining these two results we easily see that we get a
bijection between the subsets $\Gamma\subset \tDe$ and the set of $G$
orbit closures defined by associating to $\Gamma$ the intersection
$$X_{\Gamma}:=\cap_{\{D|j(\calO(D))\in\Gamma\}}D.$$
In particular $X_{\tDe}$ is the unique closed orbit while
$X=X_\vuoto$.  Let also $X_\tal = X_{\{\tal\}}$ for $\tal \in
\tDe$.

For each $\grl \in \Pi$ we choose a line bundle $\calL_\grl$ on $X$
such that $j(\calL_\grl)=\grl$ in the following way. First we
choose a basis $\calB$ of $\Pi$ and for each $\grb \in \calB$ we
take a line bundle with the required property.  Now, for $\grl
=\sum_{\grb\in\calB} c_\grb \, \grb \in \Pi$, $c_\beta\in\mathbb Z$,
we set $\calL_\grl:=\bigotimes_{\grb \in \calB} \calL_\grb^{\otimes
  c_\grb}$. We denote the restriction of these line bundles to $\XD$
by the same symbol.

If $L\subset \grL$ is a sublattice of $\grL$, then our definition allows
us to consider the graded rings
$$
R_L(X)\defi \bigoplus_{\grl\in L} \grG(X,\calL_\grl)
\;\mand\;
R_L(\XD)\defi \bigoplus_{\grl\in L} \grG(\XD,\calL_\grl).
$$
The ring $R(X)=R_\Pi(X)$ it is called the \emph{Cox ring} of $X$ and
it was studied in the case of the variety $X$ in \cite{CM1} where it
was called the ring of all sections. The fact that $G$ is simply
connected implies that each line bundle on $X$ has a canonical $G$
linearization. It follows $G$ acts on $R_L(X)$ and $R_L(\XD)$.

The space $\grG(X,\calL_\grl)$ of sections of $ \calL_\grl$ has been
described as a $G$ module in \cite{DP1} and \cite{DS}. Let us recall
here this description.

Recall that a good filtration of a $G$ module $W$ is a filtration
$W=W_0\supset W_1\supset \cdots \supset W_m=\{0\}$ by $G$ submodules
such that for each $i=1,\ldots m,$ $W_{i-1}/W_i$ is isomorphic to
$\grG(G/B,\calL_{\grl_i})$ for a suitable dominant weight $\grl_i$.

The result in \cite{DS} implies that $\grG(X,\calL_\grl)$ has a good
filtration. To be more precise first of all one shows that for any
$\lambda\in\Pi$ the map
$$\grG(X,\calL_\grl)\to \grG(\XD,\calL_\grl)$$
is surjective. 

Now for any  $\grl,\mu\in\Pi$ set $\mu \leq_\grs \grl$ if $\grl-\mu \in
\mN[\tDe]$.

Notice that, for $\tal \in \tDe$, there is a $G$ invariant section
$s_{\tal}$ of $\calL_{\tal}$, unique up to multiplication by a non
zero scalar, whose divisor is $X_\tal$.

If $\nu = \sum_\tal n_\tal \tal \geq_\grs 0$ consider  $s^\nu \defi \prod
_\tal s_\tal^{n_\tal}$. If $\grl \geq_\grs \mu$, the
multiplication by $s^{\grl-\mu}$ defines a $G$ equivariant injective
map from $\grG(X,\calL_\mu)$ to $\grG(X,\calL_\grl)$ whose image we
denote by $s^{\grl-\mu}\grG(X,\calL_\mu)$.

For any    $\nu \geq_\grs 0$ we now set
$$ F_{\grl,\nu} = \sum_{\mu \leq_\grs \grl-\nu}s^{\grl-\mu}\grG(X,\calL_\mu).$$
The $ F_{\grl,\nu}$ form  a decreasing filtration of $\grG(X,\calL_\grl)$ by
$G$ submodules. In \cite{DP1, DS} the associated graded is computed
and we have that the division by $s^{\nu}$ and restriction of sections
to $\XD$ gives an isomorphism
$  F_{\grl,\nu}/(\sum_ {\nu' >_\grs \nu} F_{\grl,\nu'}) \isocan 
V_{\grl-\nu}^*$ so that 
\begin{equation}\label{lagrada}\gr_F \grG(X,\calL_\grl)=
\bigoplus_{\mu\in \Pi^+ \mand \mu\leq_\grs \grl} s^{\grl-\mu}
V^*_\mu.\end{equation}

Clearly the filtration $F_{*,*}$ respects multiplication. This
implies that the associated graded
$$\gr_F R(X) := \bigoplus_{\grl \in \Pi} \gr_F \grG(X,\calL_\grl)$$ 
of $R(X)$ has a ring structure. Furthermore (\ref{lagrada})
gives a ring isomorphism
\begin{equation}\label{eq:grRX}
\gr_F R(X) \isocan R_\Pi(\XD)[s_{\tal_1},\dots,s_{\tal_\ell}].   
\end{equation}

In the previous section we have studied spherical weights. We want to
prove now that $\grl$ is spherical precisely when $\calL_\grl$ is spherical.

The homomorphism $\pi_X^*:\Pic(X)\to \Pic(G/\Hz)$ can be identified with
the homomorphism $\chi:\Pi\lra \grL_\Hz$ associating to $\grl
\in \Pi$, the character $\chi(\lambda)$ by which $\Hz$ acts on the
fiber of $\calL_\grl$ on the point $x_0$.

We claim that $\chi(\grl)=\chi_\Hz(\grl)$ is the dual of the character by
which $\Hz$ acts on the line fixed by $\Had$ in $V_\grl^*$ introduced
in the previous section.  To see this we may assume $\grl \in \Pi^+$.

We fix $\grl \in \Pi^+$ and $\calL=\calL_\grl$.  In this case $\calL$
has no base points over $\XD$, so, since $\XD$ is the unique closed
orbit in $X$, by Theorem \ref{teo:X} $iv)$ it also has no base points
over $X$. Thus by the reductivity of $\Had$, there is a positive
integer $m$ and a line $L\subset \Gamma(X,\calL^m)$ stable under the
action of $\Had$ and such that if $\sigma\in L-\{0\}$, $\sigma$ does
not vanish on $x_0$.  It follows that $\Hz$ acts on $L$ by the
character $-m\chi(\lambda)$.

Take the filtration $\{F_{m\lambda,\nu}\}$ of $\Gamma(X,\calL^m)$.
There is a unique submodule $ F_{m\grl,\nu}$ such that $L\subset
F_{m\grl,\nu}-\sum_ {\nu' >_\grs \nu} F_{m\grl,\nu'}$.  So $L$ has non
zero image in $V_{m\grl-\nu}^*$ and thus coincides with the unique
$\Had$ stable line in $V^*_{m\grl-\nu}$. We deduce that
$m\chi(\lambda)=\chi_\Hz(m\grl-\nu)$. Since $\nu$ lies in
  $\mZ[\tPhi].$ we have $\chi_\Hz(m\grl-\nu)=\chi_\Hz(m\grl)$, hence
$m\chi_\Hz(\grl)=m\chi(\grl)$.  Finally since $\Hz$ is connected its
character group has no torsion and we get that $\chi_\Hz(\grl)=\chi(\grl)$
as desired. We deduce the following lemma.

\begin{lem}\label{lem:PicGH} Let $\grl \in \Pi$ then
  $\pi_X^*(\calL_\grl)$ is trivial if and only if $\grl \in \Omega$.
  Moreover  if $\pi_H : G/H \lra X$ is defined by
  $\pi_H(gH)=g\cdot x_0$ then $\pi_H^*(\calL_\grl)$ is trivial if      
  and only if $\grl\in \Omega_H$.
\end{lem}
\begin{proof}
  The first claim has just been proved. As for the second it follows
  since by Theorem \ref{teo:Vust} a character $\lambda$ lies in
  $\Omega_H\cap \Pi^+$ if and only if the line in $V_\grl^*$ stable
  under $\Had$ is pointwise invariant under $H$.
\end{proof}

\subsection{Toroidal compactifications and ring of definition}
\label{ssec:toroidali} 
An embedding $Y$ of $G/H$ is called
\emph{toroidal} if there exists a basepoint preserving $G$ equivariant map $\phi :Y\lra X$.

Presently we are going to explain their construction and show that
they are defined and flat over $\R$.

Let $L_\mR= \Hom (\grL_S, \mR )$ and $L\cech_\mR = \grL_S \otimes_\mZ
\mR$ bet its dual. The $S$, or $S_H$, toric varieties are described by fans
in $L_\mR$. In particular take the cochamber $C\subset L_\mR$ of
dominant elements with respect to $\tDe$ and let $\calT_H$ be the $S_H$
toric variety associated to $C$. $\calT_H$ has a natural $\R$ form
$\mcalT_H$. In particular  in the adjoint
case $\mcalT_\Had \isocan \mA^\ell_\R$.

Choose a $\R$ form of $G$ as in Section \ref{ssec:Z}.  Consider for
any $H$ the finite field extension $\mQ(G/\Had)\subset \mQ(G/H)$.
$\mathbb Q(G/\Had)$ is the field of rational functions on $\mbX$ and
we take $\mbX_H$ equal to the normalization of $\mbX$ in $\mQ(G/H)$.
Let $\phi_H:\mbX_H\to \mbX$ denote the normalization map and let $X_H
= \mbX_H(\mk)$. 

\begin{lem} $\mbX_H$ is a projective normal and Cohen-Macaulay
  embedding of $G/H$. $\phi_H$ is a finite flat morphism. In
  particular $\mbX_H$ is proper and flat over $\R$.
\end{lem}

\begin{proof} The projectivity and normality of $\mbX_H$ are
  clear 
  from the definitions. Let us show that $\mbX_H$ is Cohen-Macaulay.
  
  To see this, let us recall $X$ is covered by the $G$ translates of
  an open set $\mathcal U$ of the form $X_S^+\times U$ where $U$ is
  the unipotent radical of the parabolic $P\subset B$ such that $\XD
  \isocan G/P$. By Theorem \ref{teo:X} we have that $X_S^+$ and $U$
  are defined over $\R$ and so is the isomorphism $X_S^+ = \calT_\Had
  \isocan \mA^\ell$.  In particular the open set $\mathcal U$ is
  defined over $\R$ and we denote by  $\boldsymbol{\mathcal U}$   the
  associated subscheme of $\mbX$ and $\mbU$ the subgroup scheme of
  $\mbG$ defining $U$.

  It easily follows that $\mbX_H$ is covered by the the $\mbG(\R)$
  translates of the preimage $\boldsymbol{\mathcal U}_H$ of $\boldsymbol{\mathcal U}$ and that
  $\boldsymbol{\mathcal U}\isocan \mcalT_H \times \mbU$.  Since $\mcalT_H$ is
  Cohen-Macaulay, also $\boldsymbol{\mathcal U}$ is Cohen-Macaulay and everything
  follows.

  Since any finite morphism between a Cohen-Macaulay scheme and a
  smooth scheme is flat we deduce that $\phi_H$ is flat and all the
  other claims are clear.
\end{proof}

We are now going to follow the method of \cite{DP2} to build all
toroidal compactifications.  For each $\tal\in \tDe$ we have already
chosen a line bundle $\mathcal L_{\tal}$ on $X$ together with a $G$
invariant section $s_{\tal}\in \Gamma(X,\mathcal L_{\tal})$. We can
then consider the vector bundle $\mathcal V:=\oplus_{\tal
  \in\tDe}\mathcal L_{\tal}$ and the $G$ invariant section $ s:=
\oplus_{\tal \in\tDe } s_{\tal}$ of $\mathcal V$. Set
$\calV^*=\{v=(v_{\tal})\in \mathcal V\st v_{\tal}\neq 0 \ \forall\
\tal\in\tDe\} $. By our previous identifications $\calV^*$ is a
principal $S_\Had$ bundle.  If $Z$ is a  $S_\Had$ variety we
can take the associate bundle $\calV^*\times_{S_\Had}Z$ on $X$ with
fiber $Z$. In particular  $\mathcal V=\calV^*\times_{S_\Had}\mathbb
A^{\ell}$, where $S_\Had$ acts on $\mathbb A^{\ell}$ via the
characters $\tal\in\tDe$.

Now take $Z$ to be a $S_\Had$ embedding over $\mathbb A^{\ell}$. The
corresponding fan $F_Z$ is a (partial) decomposition of the
fundamental Weyl cochamber $C$.  The map $Z\to \mathbb A^{\ell}$ induces
a map $\calV^*\times_{S_\Had}Z\to \mathcal V$ and we define $X_Z$ as the
fiber product
$$
\begin{CD}
  X_Z @>{s_Z}>>  \calV^*\times_{S_\Had}Z \\
  @VVV @VVV\\
  X @>{s}>>  \calV   \\
\end{CD}
$$
The $G$ action on $\calV$ preserves   $\calV^*$ and commutes with the  $S_\Had$
action. So $G$ also acts on $\calV^*\times_{S_\Had}Z$, the map
$\calV^*\times_{S_\Had}Z\to \mathcal V$ is $G$ equivariant and
$X_Z$ is a $G$ variety.

In the case of a general $G/H$ we set $X_{H,Z}$ equal to the
normalization of $X_Z$ in the field of rational function on $G/H$.
We clearly have the cartesian diagram
$$
\begin{CD}
 X_{H,Z} @>\mu_Z>> X_Z \\
@VVV @VVV\\
X_H @>\mu>> X \\
\end{CD}
$$
In particular the morphisms $\mu_Z$ and $s_{H,Z}:=s_Z\mu_Z$ are flat.
One has \cite{DP2} 

\begin{teo}\label{teo:toroidali} 
\begin{enumerate}[\indent i)]
\item Every toroidal embedding of $G/H$ is of the form $X_{H,Z}$ for
  some $S_\Had$ embedding $Z$ over $\mathbb A^{\ell}$. In particular
  it is defined and flat over $\R$.
\item $X_{H,Z}$ is complete (resp. projective) if and only if the
  projection $Z\to \mathbb A^{\ell}$ is proper (resp. projective).
\item Every $G$ orbit in $X_{H,Z}$ is of the form $\mathcal
  O_{\mathcal K}:=(s_{H,Z})^{-1}( \calV^*\times_{S_H}\mathcal K)$ for
  a unique $S_\Had$ orbit $\mathcal K$ in $Z$.
\item Let $\mathcal F_Z$ be the fan in $L_\mR$ whose cones are the
  $\tW$ translates of the cones in $F_Z$. Then $\mathcal F_Z$ is the
  fan corresponding to $S_H$ embedding $Z_H:=({X_{H,Z}})_{S_H}$.
  Furthermore each $G$ orbit in $X_{H,Z}$ intersects $Z_H$ in a unique
  $N_\Hz(S)$ orbit (notice that in accord with $iii)$ these orbits are
  in canonical bijection with $S_\Had$ orbits in $Z$).
\item The divisors in $\Delta_{X_{H,Z}}$ are defined over $\R$.
\end{enumerate}
\end{teo}

\begin{proof} All these statements are proved in \cite{DP2} in the
  case of an embedding of $G/\Had$.

  To see $i)$ in the general case take a toroidal embedding $Y$ of
  $G/H$. Let us take the quotient by the finite group scheme
  $\Had/H$. We get an embedding of $G/\Had$ which is obviously
  toroidal and hence of the form $X_Z$ for a suitable $S_\Had$
  embedding $Z$. If we now consider $X_{H,Z}$ we get a morphism $Y\to
  X_{H,Z}$ which is $G$ equivariant birational and finite. Since both
  $Y$ and $X_{H,Z}$ are normal it follows that the above morphism is a
  $G$ equivariant isomorphism.

  The proof of the remaining statements is now easy and we leave it to
  the reader.
\end{proof}

\begin{oss}\label{oss:123} 
  1) Let us point out that our result in particular implies that the
  $G$ orbits in $X_{H,Z}$ are exactly the preimages of $G$ orbits in
  $X_Z$.

  2) It is not hard to see that $X_{H,Z}$ is smooth if and only if
  $Z_H$ is smooth. Equivalently if and only if the $S_H$ embedding
  whose fan is $F_Z$ is smooth. This depends very much on the lattice
  $\Hom (\grL_{S_H}, \mathbb Z)\subset L_\mR$.

  3) There exists an open affine covering $\{U_i=\Spec R_i\}$ of the
  $\R$ form of $X_{H,Z}$ such that $R_i$ are $\R$ free modules and
  $U_i\cap U_j = \Spec R_{ij}$ where $R_{ij}$ is also a $\R$ free
  module.
\end{oss}

\subsection{Line bundles on a toroidal embedding} In this section we
assume $Y$ to be a smooth toroidal compactification of $G/H$ with the
$\R$ structure described in the previous section.

We have the following Lemma about the structure of the Picard group of
$Y$.

\begin{lem}
  \label{lem:PicY}Let $Y$ be a equivariant smooth toroidal
  compactification of $G/H$ then
  \begin{enumerate}[\indent i)]
  \item We have the following sequence describing the Picard group of
    $Y$:
    \begin{equation*}\label{ilpicardo}
      \begin{CD}
        0 @>>> \bigoplus_{D\in \grD_Y} \mZ \calO(D) @>>> \Pic(Y)
        @>{\mj_Y^*}>> \grL_H @>>> 0.
      \end{CD}
    \end{equation*}
  \item $\SPic_0(Y) = \bigoplus_{D\in \grD_Y} \mZ \calO(D)$.
  \item For each closed $G$ orbit $O$ of $Y$ consider the restriction
    $\mi_O^* : \SPic_0(Y) \lra \Pic(O)$ of line bundles to $O$. Then
    the product of these restriction maps $\mi^*: \SPic_0(Y) \lra
    \prod \Pic(O)$ is injective.
  \item All line bundles on $Y$ are defined and flat over $\R$.
\end{enumerate}
\end{lem}

\begin{proof}
  The only thing we need to show to prove $i)$ is the injectivity of
  the map from $\bigoplus_{D\in \grD_Y} \mZ \calO(D)$ to $\Pic(Y)$.
  Notice that, since $G$ is semisimple and simply connected and
  $\Pic(Y)$ discrete, every line bundle has a unique $G$
  linearization. Thus $\Pic(Y)\isocan \Pic_G(Y)$. It follows that it
  is enough to prove the injectivity of the map from $\bigoplus_{D\in
    \grD_Y} \mZ \calO(D)$ to $\Pic_G(Y)$. Consider the restriction map
  $\Pic_G(Y)\to\Pic_{S_H}(Y^+_S)$. $\Pic_{S_H}(Y^+_S)$ is isomorphic
  to $\bigoplus \mZ\calO(D')$ where the sum is take over all $S_H$
  equivariant divisors $D'$. So the claim follows from Theorem
  \ref{teo:toroidali}.  Since $\SPic_0 $ is the kernel of $\mj_Y^*$
  this also proves $ii)$.
  
  $iii)$ follows from the previous considerations and the fact that,
  up to isomorphism a $S_H$ equivariant line bundle on $Y^+_S$ is
  completely determined by its restriction to the closed orbits, that
  is the $S_H$ fixpoints in $Y^+_S$.

  Finally let $D \in \Delta _Y$.   
  By Theorem \ref{teo:toroidali} v) it is defined over
  $\R$. We know that $\grL_H=\Pic(G/H)$ is generated
  by the codimension $1$ irreducible $B$ orbits in $G/H$ and that
  these orbits are defined over $\R$ by Lemma 2.7 in \cite{DS} and
  Lemma \ref{lem:STZ}. Thus $iv)$ follows from $i)$.
\end{proof}

Let $F_Y$ be the fan associated to the toric variety $Y^+_S$ and let
$F_Y(i)$ be the set of faces of $F_Y$ of dimension $i$. In particular
the closed orbits of $Y$ are parametrized by $F_Y(\ell)$ while
$F_Y(1)$ can be identified with $\Delta_Y$ the set of $G$ invariant
divisors.  For each $\rho \in F_Y$ we set $y_\rho: = y_\eta$ for
$\eta$ a generic element in $\rho$ and denote by $O_\rho = G\,y_\rho$
the associated $G$ orbit.

By Theorem \ref{teo:toroidali} and the description of the equivariant
Picard group of a toric variety we have the following description of
the strictly spherical line bundles on $Y$:
\begin{equation}\label{ilpicdiT}
  \SPic_0(Y)=\{\bgrl = (\grl_\tau) \in \prod_{\tau \in F_Y(\ell)} \Omega_H
  \st \grl_\tau = \grl_{\tau'} \text{ on } \tau \cap \tau'\}.
\end{equation}

We can think of $\bgrl$ as a real valued function on the Weyl
cochamber $C$ which coincides with the linear form $\grl_\tau$ on the
face $\tau$.  We denote by $\calL_\bgrl$ a line bundle whose class is
given by $\bgrl$.  In particular we can describe in this way the line
bundles $\calO(D)$ for each divisor $D\in \grD_Y$.  Indeed let $v_D
\in \grL\cech_{S_H}$ be a non divisible element of $\grL_{S_H}$ in the
$1$ dimensional face of $F_Y$ associated to $D$. For each $\tau \in
F_Y(\ell)$ notice that, since $Y$ is smooth, the set $\{v_{D'} \st D'
\in \grD_Y\} \cap \tau$ is a basis of $\grL\cech_{S_H}$.

So we can define $\gra_{D,\tau} \in \grL_{S_H}$ to be the weight which
is equal to zero if $v_D \notin \tau$ while if $v_D \in \tau$ it is
$1$ on $v_D$ and zero on each $v_{D'}\in \tau$ with $D'\neq D$. It is
then easy to see that $\bgra_D = (\gra_{D,\tau})_{\tau \in F_Y(\ell)}$
is the class of $\calO(D)$ in $ \SPic_0(Y)$.  \medskip

Now we want to describe the sections of a strictly spherical line
bundle on $Y$ in the case of characteristic $0$. The proofs are very
similar to the one given in \cite{Bifet}. A description of the section
of a line bundle on a general spherical variety is given in
\cite{Brion} and we could have used that result as well. However the
description we are going to give is more suited to our purpose.

For $D \in \grD_Y$ let $s_D$ be a $G$ invariant section of
$\grG(Y,\calO(D))$ vanishing on $D$. If $\bgrl = \sum_D a_D \bgra_D$
we set $s^\bgrl = \prod_D s_D^{a_D}$. Also for a given $\mu \in
\Omega^+_H$ consider the line bundle $\phi^*(\calL_\mu)$ where
$\phi:Y\lra X$ is the $G$ equivariant projection from $Y$ to $X$.
This line bundle corresponds to the element $\bmu \in \SPic_0(Y)$ with
$\mu_\tau = \mu$ for all $\tau\in F_Y(\ell)$ under the identification
of $\Omega_H$ with $\grL_{S_H}$ given by Lemma \ref{lem:sferici} v).
In particular $V_\mu^*$ is a submodule of $\grG(Y,\phi^*(\calL_\mu))$.

For $\bgrl \in \SPic_0(Y)$ set
\begin{align*}
\calA(\bgrl) &= \{\mu \in \Omega_H^+ \st \forall \tau \in
F_Y(\ell),\  \grl_\tau - \mu =  \sum _{v_D
  \in \tau} a_D \bgra_D \text{ with } a_D \geq 0 \} \\
    & = \{ \mu \in \Omega_H^+ \st \mu \leq \bgrl \text{ on } C\}
\end{align*}

We then have the following Theorem whose proof is completely analogous  to the one
given in \cite{Bifet}.
\begin{teo}\label{teo:sezioniY}Assume $Y$ to be a smooth toroidal
  compactification of $G/H$ and assume the field to be of
  characteristic zero and let $\bgrl \in \SPic_0(Y)$ then
$$
\grG(Y,\calL_\bgrl) = \bigoplus_{\mu \in \calA(\bgrl)} s^{\bgrl - \bmu}V_\mu^*.
$$
\end{teo}

From the above result we can deduce, as in \cite{Bifet} \S 4.2 and
\cite{Ruzzi} the following Corollary.  Let $\Omega_H^{++}$ be the set
of elements of $\Omega_H$ that are in the interior of the Weyl cochamber
$C$.

\begin{cor}\label{cor:sezioniY}Let $Y$ be a smooth toroidal compactification of
  $G/H$, let $\bgrl \in \SPic_0(Y)$ and assume the field to be of
  characteristic zero. Then
  \begin{enumerate}[\indent i)]
  \item For every $\mu \in \grL^+$  $V_\mu^*$ is an irreducible summand
    of $\grG(Y,\calL)$ if and only if
    $\grG(Y_S,\calL\ristretto_{Y_S})$ has a section of $S_H$ weight
    equal to $\mu$.
  \item For every line bundle $\calL$ generated by global sections,
    the restriction map $$\grG(Y,\calL)\to
    \grG(Y_S,\calL\ristretto_{Y_S})$$ is surjective.
  \item $\calL_\bgrl$ is an ample line bundle if and only if it is
    very ample.
  \item $\calL_\bgrl$ is an ample line bundle if and only if
    $\grl_\tau \in \Omega_H^{++}$ and $\grl_\tau < \grl_{\tau'}$ on
    $\tau'\senza \tau$ for all faces $\tau$ and $\tau'$ of $F_Y$ of
    maximal dimension.
  \end{enumerate}
\end{cor}

\begin{oss}
We have limited our discussion to strictly spherical line bundle and
to characteristic 0. Using Frobenius splitting methods it is easy  to  generalize
the previous results as in \cite{DS}. However the stated result are
enough for our purpose here. 
\end{oss}

\section{The quotient of a symmetric variety}\label{sec:GIT}
Let $Y$ be an embedding of $G/H$ and let $K$ be a subgroup such that
$\Hz\subset K \subset H$. Any line bundle on $Y$ has a $G$
linearization, so in particular it has a $K$ linearization.  Recall
that if $\calL$ is an ample line bundle on $Y$ a point $y$ on $Y$ is
said to be $\calL$ semistable (with respect to the action of $K$) if for some 
 $n>0$ there exists $f\in \Gamma (Y,\calL^n)^K$ such that $f(y)\neq
0$.  We denote by $Y^{ss}(\calL)$ the set of $\calL$ semistable
points, or in case $\calL$ is chosen just semistable points.
$Y^{ss}(\calL)$ is a possibly empty open subset of $Y$.  By
\cite{Newstead} Theorem 3.21 there exists a good quotient of the set
of $\calL$ semistable points which we shall denote by $K\lGIT_\calL
Y$.

In this section we are going to prove the following Theorem

\begin{teo}\label{teo:GIT}
  Let $Y$ be an equivariant projective embedding of $G/H$ and let $\calL$ be an
  ample and spherical line bundle on $Y$ then the inclusion
  $Y_S\subset Y$ induces an isomorphism of algebraic varieties between
  $\tW\backslash Y_S$ and $K\lGIT_\calL Y$.\end{teo}

We will prove this Theorem by computing the invariant sections. We
will analyze first the case of the wonderful compactification and the
case of the quotient of the open affine part $G/H$.

\subsection{Invariants and semiinvariants of the Cox ring of the
  wonderful compactification} In this section we compute the $H$
invariants of the ring $R(X)$. We use the notations introduced in
section \ref{ssec:X}.

\begin{lem}\label{lem:grGammaX}
  Let $\grl \in \Pi$. Then $(\gr_F(\grG(X,\calL_\grl)))^{H}
  =\gr_F(\grG(X,\calL_\grl)^{H})$. In particular the dimension of
  the space of invariants $\grG(X,\calL_\grl)^{H}$ equals the
  cardinality of the set $K_\lambda:=\{\mu \in \Omega^+ \st \mu
  \leq_\grs \grl\}$ if $\grl \in \Omega_H$ and is zero otherwise.
\end{lem}

\begin{proof}
  In characteristic zero the equality
  $$(\gr_F(\grG(X,\calL_\grl)))^{H} =\gr_F(\grG(X,\calL_\grl)^{H})$$ 
  is an immediate consequence of the linear reductivity of $H$.

 Also (in arbitrary characteristic) by equation
  \eqref{lagrada} we have that $$(\gr_F(\grG(X,\calL_\grl)))^{H} =
  \bigoplus_{\mu\in \Pi^+ \mand \mu\leq_\grs \grl} s^{\grl-\mu}
  (V^*_\mu)^H.$$ By Vust criterion (Theorem \ref{teo:Vust}) $(V_\mu^*)^H$ is
  one dimensional if $\mu \in \Omega_H^+$ and it is zero otherwise. So,
  since by Lemma \ref{lem:sferici} $\mZ[\tPhi]\subset \Omega_H$ we
  have that $(\gr_F(\grG(X,\calL_\grl)))^{H}$ has dimension equal to
  $|K_\lambda|$ if $\grl \in \Omega_H$ and it is zero otherwise.
  
  In general $(\gr_F(\grG(X,\calL_\grl)))^{H} \supset
  \gr_F(\grG(X,\calL_\grl)^{H})$ so
  $$\dim\grG(X,\calL_\grl)^{H}\leq\dim (\gr_F(\grG(X,\calL_\grl)))^{H}.$$

  On the other hand by Theorem \ref{teo:X} and Lemma \ref{lem:PicY}
  the variety $X$ and the spaces $\grG(X,\calL_\grl)$ are all defined
  over $\R$. Lemma \ref{lem:invsuZ} then clearly implies that in
  positive characteristic $\dim\grG(X,\calL_\grl)^{H}$ can only
  increase.  This together with the previous inequality implies our
  claim.
\end{proof}

We compute now the ring $R(X)^{H}$. By Lemma \ref{lem:grGammaX}, for
each $\tal \in \tDe$ we can choose $p_\tal \in
\grG(X,\calL_{\tom_\tal})$ an $\Hz$ invariant section which does not
vanish on $\XD$. So, if $\grl=\sum a_\tal \tom_\tal\in \Omega^+$, we
can define
\begin{equation}\label{linvario}p^{\grl}= \prod_{\tal \in
    \tDe}p_\tal^{a_\tal}.\end{equation}

\begin{prp} \label{prp:invX} The set $\{s^{\mu}p^{\grl} \st \mu \in
  \Pi \mand \mu \geq_\grs 0 \mand \grl \in \Omega_H^+\}$ is a $\mk$
  basis of $R(X)^{H}$. In particular the ring $R(X)^{\Hz}$ is a
  polynomial ring in the variables $s_\tal, p_\tal$ with $\tal \in
  \tDe$.
\end{prp}
\begin{proof}
  Notice first that by Lemma \ref{lem:grGammaX}, if $\grl \in
  \Omega_H$, $\grG(X,\calL_\grl)^{H} = \grG(X,\calL_\grl)^{\Hz}$ so
  it is enough to prove the claim in the case of $\Hz$.

  The image of $p_\tal$ in the graded ring $\gr_F(R(X))$ defines an
  $\Hz$ invariant element $\bar p_\tal$ of $V^*_{\tom_\tal}$.  So by the
  description of the $\Hz$ invariants of $\gr_F(R(X))$ the image of
  the elements $s^\mu p^\grl$ in the graded ring $\gr_F(R(X))$ is a
  $\mk$ basis of the space of $\Hz$ invariants. This implies that the
  elements $s^\mu p^\grl$ are linearly independent.

  By construction, the elements $s^\mu p^\grl$ are $\Hz$ invariants.
  So, again by Lemma \ref{lem:grGammaX} they are a $\mk$ basis of
  $R(X)^{\Hz}$.
\end{proof}

The computation of semi invariants is similar. If $V$ is a
representation of $\Had$ we denote by $V_{si}^{\Had}$ the
subspace spanned by the set of semi invariant vectors i.e. vectors
spanning lines fixed by $\Had$.

By Theorem \ref{teo:DS}, we know that there are semi invariants which
are not $\Hz$ invariants only if there exists an exceptional simple root.
Set $\grD_e =\{\gra \in \grD_1 \st \gra$ is exceptional$\}$ and
$\tDe_{ne}=\{\tal \in \tDe \st \alpha$ is not exceptional$\}$. By
Theorem \ref{teo:DS} the set $\{\om_\gra\st \gra \in
\grD_e\}\cup\{\tom_\tal \st \tal \in \tDe_{ne}\}$ is a basis of $\Pi$.
Let $\bar q _\gra \in V_{\om_\gra}^*$ be a non zero $\Had$ semi
invariant. $\bar q _\gra$ is unique up to multiplication by a non zero
scalar. So the ring $\big(R_\Pi(\XD)\big)_{si}^{ \Had}$ of semi
invariants is a polynomial ring in the generators $\bar q_\gra$ with
$\gra \in \grD_e$ and $\bar p_\tal$ (the restriction of $p_\tal$ to
$\XD$) with $\tal \in \tDe_{ne}$.  Using Corollary \ref{cor:invV} and
arguing as in Lemma \ref{lem:grGammaX} we deduce that there exists
$q_\gra\in \grG(X,\calL_{\omega_\gra})_{si}^{ \Had}$ such that its
restriction to $\XD$ is equal to $\bar q_\gra$. If $\grl =
\sum_{\gra\in\grD_e} c_\gra \om_\gra + \sum_{\tal\in \tDe_{ne}} c_\tal
\tom_\tal \in \Pi^+$, we define $q^\grl=\prod_{\gra\in \grD_e}
q_\gra^{c_\gra} \,\cdot \prod_{\tal \in \tDe_{ne}}p_\tal^{c_\tal}$.
The arguments given in the case of invariants can now be easily
adapted implying
\begin{prp}\label{prp:semiinvX}
  Let $\grl \in \Pi$. Then $(\gr_F(\grG(X,\calL_\grl)))_{si}^{ \Had}
  =\gr_F(\grG(X,\calL_\grl)_{si}^{ \Had})$.  Moreover the set
  $\{s^{\mu}q^{\grl} \st \mu \in \Pi \mand \mu \geq_\grs 0 \mand \grl \in
  \Pi^+\}$ is a $\mk$ basis of $R(X)_{si}^{ \Had}$. In particular
  the ring $R(X)_{si}^{ \Had}$ is a polynomial ring in the variables
  $s_\tal$ with $\tal \in \tDe$, $p_\tal$ with $\tal \in \tDe_{ne}$
  and $q_\gra$ with $\gra \in \grD_e$.
\end{prp}

\subsection{Filtration of the coordinate ring of $G/H$ and Richardson Theorem}
\label{ssec:inv} We want now to use the wonderful variety
$X$ to define a filtration of the coordinate ring of $G/H$. In the
case in which $H$ is the diagonal subgroup in $G=H\times H$, these
ideas already appear in \cite{Strick}.  This will be used to describe
the $H$ invariants of this ring. This description of the invariants
has already been given by Richardson \cite{Rich} but a proof in our
setting seems natural.

First we make explicit the relation between the coordinate ring of
$G/H$ and the ring $R_{\Omega_H}(X)$.

For each $\tal \in \tDe$ we  choose a trivialization
$\grf_{\tom_\tal}:\pi^*(\calL_{\tom_\tal})\lra \calO_{G/\Hz}$. Given
$\grl=\sum_{\tal \in \tDe} c_\tal \tom_\tal \in \Omega$ we obtain the
trivialization of $\pi^*(\calL_\grl)$ given by $\otimes_{\tal \in
  \tDe} \grf_{\tom_\tal}^{\otimes c_\tal}$. With these choices the pull
back of sections defines a ring homomorphism:
$$\pi_H^*:R_{\Omega_H}(X)\lra \mk[G/H].$$
Notice also that since $s_\tal$ is $G$ invariant the functions
$\pi_H^*(s_\tal)$ are constant and we can normalize them to be equal to
$1$.

For each $\lambda\in\Omega_H$ we consider the $G$ submodule $$\mF_\grl
:= \pi_H^*(\grG(X,\calL_\grl)) $$ of $\mk[G/H]$. Notice that since
the $s_{\alpha}$ are all equal to $1$, we clearly have that if $\mu
<_\grs \grl$, $\mF_\mu\subset \mF_\grl $.  Also, since the image of
$\pi_H$ is dense in $X$ we have that $\pi_H^*$ restricted to
$\grG(X,\calL_\grl)$ is an isomorphism onto $\mF_\grl$. Furthermore if
we set $ \mF'_\grl=\sum_{\mu <_\grs \grl} F_\mu, $ we have
$\mF_\grl/\mF'_\grl \isocan V_\grl^*$.

\begin{prp}
  \label{prp:XGH} The map $\pi_H^*$ induces an isomorphism of rings
$$\grf:\frac{R_{\Omega_H}(X)}{(s_\tal-1 \st \tal \in \tDe)} \lra \mk[G/H].$$
\end{prp}
\begin{proof} $\grf$ is clearly well defined and its surjectivity
  follows immediately from Proposition \ref{lem:splitta}.
    
  Let us now show that $\grf$ is injective. As above set $\Xi_H =
  \Omega_H/\Omega_\Had$ and for all cosets $\xi \in \Xi_H$ define $R_\xi =
  \bigoplus_{\grl \in \xi} \grG(X,\calL_\grl)$ so that
  $R_{\Omega_H}(X)=\bigoplus_{\xi\in \Xi_H}R_\xi$ is a $\Xi_H$ grading
  of the ring $R_{\Omega_H}(X)$.
  
  On the other hand by Proposition \ref {lem:splitta}, the coordinate ring
  $\mk[G/H]$ decomposes as the direct sum $\oplus_{\xi\in \Xi_H}
  \grG(G/\Had,\calL_\xi)$ and the restriction of $\pi^*_H$ decomposes
  as the direct sum $\oplus_{\xi\in \Xi_H} \mj^*_\xi$ where
  $\mj^*_{\xi}:R_\xi\to \grG(G/\Had,\calL_\xi)$ is induced by the
  inclusion $\mj_X$ of $G/\Had$ in $X$.
  
  Also, since the elements $\{s_\tal-1\st \tal\in \tDe\}$ lie in $R_0$
  the ideal $I$ that they generate decomposes as the direct sum
  $I=\oplus_{\xi\in \Xi_H} I_\xi$ with $I_\xi=I\cap R_\xi$, for each
  $\xi\in \Xi_H$.  Thus $\mj^*_\xi$ induces a map
  $$\grf_\xi:R_\xi/I_\xi\to \grG(G/\Had,\calL_\xi)$$ and it is enough
  to see that $\grf_\xi$ is injective for each $\xi\in \Xi_q$.

  Fix $\xi\in \Xi_H$.  Let $g=\sum_{\grl \in A} g_\grl \in R_\xi$ with
  $g_\grl \in \grG(X,\calL_\grl)$ and $A$  a finite subset of the
  coset $\xi$.  Assume $\pi_H^*(g)=0$.  By assumption there exists
  $\mu \in \xi$ such that $\mu \geq_\grs \grl$ for all $\grl\in A$.
  Set $g' = \sum _{\grl\in A} s^{\mu-\grl}g_\grl$ and notice that
  $g'\coinc g \mod I_\xi$ and that $g'\in \grG(X,\calL_\mu)$. We have 
  $\pi_H^*(g')=\pi_H^*(g)=0$ and since $\pi_H^*$ restricted to
  $\grG(X,\calL_\mu)$ is injective, $g'=0$ and $g \in I_\xi$ as
  desired.
\end{proof}

\begin{cor} The $G$ submodules $\mF_\grl$, $\lambda\in\Omega_H$, induce 
  a good (increasing) filtration of the coordinate ring $\mk[G/H]$.
\end{cor}

We are now going to use this filtration to study the ring of
invariants $\mk[G/H]^K$. We  first need a well known lemma.

\begin{lem}\label{lem:pesoalto} Fix a dominant weight
  $\lambda\in\Omega^+$.

  \begin{enumerate}[\indent i)]
  \item Let $\phi \in V_\grl^*$ be a nonzero $\Hz$ invariant and let
    us consider the decomposition of $V_\grl^*$ with respect to the
    action of $T$. Then the lowest weight component of $\phi$ is not
    zero;
  \item Let $p^\grl \in \grG(X,\calL_\grl)$ denote the $\Hz$ invariant
    defined in formula (\ref{linvario}) and consider the decomposition
    of $p^\grl \in \grG(X,\calL_\grl)$ with respect to the action of
    $T$.  Then the lowest weight component of $p^\grl$ is not zero.
  \end{enumerate}
\end{lem}

\begin{proof}   In $V_\grl^*$ there is a non zero vector $v$ fixed
  by the maximal unipotent subgroup $U^{-}$ opposite to $B$. This
  vector is unique up to a non zero scalar and has weight $-\grl$
  which is the lowest weight of $V_\grl^*$. Write $\phi= w+av$ with
  $w$ lying in the unique $T$ stable complement $V'$ of the one
  dimensional space spanned by $v$ and $a\in \mk$. $V'$ is $B$ stable. We need to prove
  $a\neq 0$.

Consider the $G$ submodule $W$ generated by $\phi$. Since $W$
  must contain a non zero vector fixed by $U^-$, it has to contain
  $v$.

  On the other hand $B \cdot \Hz$ is dense in $G$ so the subspace $W$
  is equal to the space spanned by the vectors $b\cdot\phi$ with $b\in
  B$.  If $a=0$ then $W$ would be contained in $V'$ giving a
  contradiction. This proves $i)$.
  
 To see  $ii)$  it suffices to  consider  the image $\bar p ^\grl$ in
  $\grG(X,\calL_\grl)/F'_{\grl,\grl}\isocan V_\grl^*$ which is non
  zero by the very definition of $p^\grl$.
\end{proof}

We can now prove Richardson Theorem (see \cite{Rich} Corollary 11.5).
Notice that, since $N_\Hz(S)\subset \Hz \subset K$ the inclusion of $S_H$ in
$G/H$ induces a map from $\tW\backslash S_H$ to $K
\backslash\!\backslash G/H$.

\begin{teo}\label{teo:Rich} Let $\Hz \subset K \subset H$ be a 
  subgroup of $H$. Then the inclusion $S_H\subset G/H$ induces an
  isomorphism $\tW\backslash S_H \isocan K \backslash\!\backslash
  G/H$.
\end{teo}

\begin{proof}
  By the definition of $\tW$, the restriction of functions from $G/H$
  to $S_H$ induces a homomorphism
 $$d:\mk[G/H]^K\to \mk[S_H]^{\tW}.$$
 We claim that $d$ is an isomorphism.

 To see this we first make some remarks on the $K$ invariants of
 $\mk[G/H]$.  For $\grl \in \Omega_H^+$ let $f^\grl :=
 \pi_H^*(p^\grl)$. Arguing as in Proposition \ref{prp:invX} it is
 easy to see that the elements $f^\grl$ with $\grl \in \Omega_H^+$ are
 a basis of $\mk[G/H]^K$ as a vector space. In particular for each
 $\lambda\in \Omega_H^+$, $\mF_\grl^{K}/ {\mF'}_\grl^{K}$ is one
 dimensional and spanned by the class of $f^\grl$ (notice that
 $\Omega_H\subset \Omega_K$).

 The computation of the $\tW$ invariants of the ring $\mk[S_H]$ is
 also very simple. Let $\mk[S_H]=\bigoplus_{\grl \in \grL_{S_H}} \mk
 \grf_\grl$ where $\grf_\grl$ is a function of weight $\grl$. We know
 by Theorem \ref{teo:Vust} and Lemma \ref{lem:sferici} that the
 restriction $r$ of character from $T$ to $S$ induces an isomorphism
 between $\grL_{S_H}$ and $\Omega_H$ so we identify the two lattices.
 Also a weight $\lambda\in \Omega_H$ is dominant with respect to
 $\tDe$ if and only if it is dominant with respect to $\Delta$.  If
 $\grl \in \grL_{S_H}^+$ we set
  $$\psi_\grl=\sum_{\psi \in \tW \cdot \grf_{-\grl}} \psi.$$
 The elements $\psi_\grl$ with $\grl\in \Omega_H^+$ are clearly a
 basis of $\mk[S_H]^\tW$.
    
 Given $\grl\in \Omega_H^+$, let $U_\grl$ denote the span of elements
 $\psi_\mu$ with $\mu \in \Omega_H^+$ and $\mu \leq_\grs \grl$.

 Notice that $f^\grl_S:=d(f^\grl)$ lies in $U_\grl$. Indeed $f^\grl_S$
 is a $\tW$ invariant and its weights are in a subset of the weights
 appearing in $V_\grl^*$. Thus for each $\lambda\in \Omega_H^+$,
 $d(\mF_\grl)^K\subseteq U_\grl$. We claim that $d$ maps isomorphically
 $\mF_\grl^{K}$ onto $U_\grl$. This will imply our claim.
  
 By an easy induction we need  to show that $f^\grl_S \notin
 \sum_{\mu<_\grs\grl \mand \mu \in \Omega_H^+} U_\mu$.  Using Lemma
 \ref{lem:pesoalto} it suffices to prove that the restriction to $S_H$
 of a lowest weight vector $h$ in $\mF_\grl$ is non zero.  The closure
 of $S_H$ in $X$ contains the unique point of the closed orbit $\XD$
 fixed by $B$.  $h$ does not vanish at this point. Since $h$ is non
 zero at a point in the closure of $S_H$ it cannot vanish on $S_H$
 proving that $f^\grl_S$ does not lie in $\sum_{\mu<_\grs\grl \mand
   \mu \in \Omega_H^+} U_\mu$.
\end{proof}

\begin{oss}\label{laliber}
  Notice that in particular $K\lGIT G/H$ does not depend on the choice
  of the subgroup $K$ between $\Hz$ and $H$. However it is not true in
  general that the $K$ orbits in $G/H$ are the same of the $H$ orbits
  in $G/K$. To see this it is enough to take $G=SL(2,\mC)$ and $\grs$
  the conjugation by $\left( \begin{smallmatrix} i & 0 \\ 0 & -i
    \end{smallmatrix} \right)$. Then $\Hz$ is the diagonal torus and
  it is easy to check that $\Hz \left( \begin{smallmatrix} 1 & 1 \\ 0
      & 1 \end{smallmatrix} \right)\Had \neq \Had \left(
    \begin{smallmatrix} 1 & 1 \\ 0 & 1 \end{smallmatrix} \right)\Had$.
\end{oss}

To complete our picture we show, with a different proof,  another result of Richardson which tells us   the orbits of the elements in $S_H$ are precisely the closed orbits in $G/H$. 

\begin{prp}\label{prp:closed}Let $\Hz \subset H,K \subset \Had$.
  Let $s \in S_H$ then $K s$ is closed in $G/H$.
\end{prp}

\begin{proof}
  Since $\Hz$ has finite index both in $H$ and in $K$, it is enough to study the
  case $H=K=\Hz$. 

  Fix $V$ to be a finite dimensional faithful representation of $G$.
  Consider the map $\chi : G/\Hz \lra G \subset GL(V)$ given
  by $\chi(g\Hz)=g\grs(g)^{-1}$. By \cite{Springer} Theorem 5.4.4 this
  is a closed immersion. Notice that $\chi(h\cdot x) = h\chi(x)h^{-1}$
  for all $h\in \Hz$ and $x\in G/\Hz$.

For $s\in S_H$ set $x = \chi(s)=s^2$. The element $s$   is   semisimple   in$GL(V)$. We
  want to prove that the orbit $\{hxh^{-1}\st h\in \Hz\}$ is closed in
  $GL(V)$. 
  
  Let $p(t)$ be the minimal polynomial of $x$. Since $x$ is
  semisimple $p(t)$ does not have  multiple roots. For all $\grl \in \mk^*$
  and $y \in G$ set $V_\grl(y)=\{v \in \gog \st Ad_y(v)=\grl v\}$.
  Notice that for $y\in T$
  $$V_\grl(y)=\begin{cases}
    \bigoplus_{\gra\in \Phi \st \alpha(y)=\grl} \gog_\gra \ \ \ \ \text{if\ } \grl \neq 1\\ 
   \got \oplus \bigoplus_{\gra\in \Phi \st
    \alpha(y)=1} \gog_\gra \ \ \ \ \text{if\ } \grl=1.
    \end{cases}
    $$
   Given  $\grl \in \mk^*$ and $\mu =
  \tfrac{1}{2}(\grl +\grl^{-1})$ set 
  
 $$W_\mu (y)=\begin{cases}V_\grl(y)\oplus
  V_{\grl^{-1}}(y)\ \ \ \ \text{if\ }  \grl \neq \pm 1\\ V_\grl(y)
\ \ \ \ \text{if\ } \grl = \pm 1.\end{cases}$$
 Notice that we have
  \begin{equation}
    \label{eq:hW}
\goh = \bigoplus_\mu (\goh \cap W_\mu(x))    
  \end{equation}
  since $\goh = \got^\grs \oplus \bigoplus _{\gra \in \Phi_0}\gog_\gra
  \oplus \bigoplus _{\gra \in \Phi_1^+}\mk(x_\gra+\grs(x_\gra))$ and
  $\grs(\gra) (x)= \gra(x)^{-1}$ so that $x_\gra+\grs(x_\gra) \in
  W_{\gra(x)+\gra(x)^{-1}}$. Let $n_\mu = \dim (\goh \cap W_\mu(x))$.  We define the closed subset
  $M\subset G$ as follows. An element $y\in G$ lies in $M$ if
  \begin{enumerate}[\indent i)]
  \item $p(y) = 0$
  \item  the characteristic polynomial of $Ad_y$
                   is 
                  equal to the characteristic polynomial of
            $Ad_x $
                  \item$ \dim (\goh \cap W_\mu(y))\geq n_\mu$ for all 
         $ \mu.$
\end{enumerate}
By condition iii) and relation \eqref{eq:hW} it follows that if $y\in M$,
  $\dim (\goh \cap W_\mu(y) )=  n_\mu$ for all $\mu$. In
particular $\dim \goh \cap V_1(y) = \dim \goh \cap V_1(x)$ so that
$\dim \goh \cap Z_\gog(y)= \dim \goh \cap Z_\gog(x)$. 

Now by condition
i) if $y\in M$, $y$ is semisimple and so by 
\cite{Springer} Theorem 5.4.4 we have that $Z_\gog(y)$ is the Lie
algebra of $Z_G(y)$. It follows that   $\dim \goh \cap Z_\gog(y) = \dim H \cap
Z_G(y)= n_1$ and finally $\dim \Hz y = \dim H - n_1$ so that
 every $\Hz$ orbit in   $M$ has the same dimension.
 
  In
particular every $\Hz$ orbit in $M$ is closed. Since $x \in M$  the Proposition follows.
\end{proof}
\begin{oss}\label{ilridutt} Notice that our proof works also under the slightly more general assumption that $G$ is reductive. This will be useful later on.
\end{oss}

\begin{oss}We notice that if $\Hz \subset K \subset H$ and $s\in S$
  then $Ks = Hs$ in $G/H$. Indeed by Remark \ref{laliber} we have that
  $K\backslash\backslash G/H = H\backslash\backslash G/H$ so the natural map $Kx\mapsto Hx$ from the
  set of $K$ orbits into the set of $H$ orbits is a bijection at the
  level of closed orbits. At this point  everything follows from the fact that
  $Hs$ is a union of closed $K$ orbits.
\end{oss}

\subsection{The quotient of a smooth projective toroidal compactification}\label{ssec:GITtoroidale}
Let us now assume that $\calL$ is a line bundle generated by global
sections but not necessarily ample on a given smooth, projective
toroidal compactification $Y$ of $G/H$. This is going to be useful in
order to prove Theorem \ref{teo:GIT} in its full generality.

  Since $\calL$ is spherical, $\pi_Y^*(\calL)$ is
trivial.  Notice now that the $G$ linearization of $\calL$ restricts
to a $N_\Hz(S)$ linearization of $\calL\ristretto_{Y_S}$ and that
$Z_\Hz(S)$ acts trivially on $\calL\ristretto_{Y_S}$. To see this first
notice that $Z_\Hz(S)$ acts trivially on the fiber of $\calL$ over
$y_0$. Indeed $\pi_Y^*(\calL)$ is trivial, hence $\Hz$ acts trivially on
the fiber of $\calL$ over $y_0$ and $Z_\Hz(S)\subset \Hz$.  Now,
$Z_\Hz(S)$ commutes with $S$ and $S\cdot y_0$ is dense in $Y_S$ so
necessarily $Z_\Hz(S)$ acts trivially on $\calL\ristretto_{Y_S}$.

Notice that $Y^{ss}(\calL)= Y^{ss}(\calL^n)$ for any $n>0$, so
$K\lGIT_{\calL^n} Y\isocan K\lGIT_\calL Y$. Since
$\Omega/\Omega_K$ is finite, up to taking a power of $\calL$ we can
always assume that $\pi_{Y,K}^*(\calL)\isocan \calO_{G/K}$ where
$\pi_{Y,K}:G/K\lra Y$ is defined by $\pi_{Y,K}(gK)=g\cdot y_0$. We call
these line bundles $K$ spherical.  We have
\begin{teo}\label{glinvarianti}  Let $\calL$ be a $K$ spherical line
  bundle on $Y$. Then 
$$\grG(Y,\calL)^K \isocan \grG(Y_S,\calL\ristretto_{Y_S})^\tW.$$
\end{teo}
 
We need first a well known general fact on flat schemes over $\R$ whose proof we give for completeness.

\begin{lem}\label{lem:piattezza}
  Let $U$ be a projective flat scheme over $\R$ such that there exists
  an open affine covering $\{U_1,\dots,U_n\}$ of $U$ with $U_i=\Spec
  R_i$ and $U_i\cap U_j = \Spec R_{ij}$ where $R_i$ and $R_{ij}$ are
  free $\R$ modules. Let $\calL$ be a locally free sheaf over $U$.  For each ring extension $\R\lra B$ let $U_B = U
  \times_{\Spec(\R)} \Spec(B)$ and $\calL_B$ the pull back of $\calL$ to $U_B$.
  Then:
  \begin{enumerate}[\indent i)]
  \item $\grG(U,\calL)$ is a finitely generated  free $\R$ module;
  \item the map $B\otimes_\R \grG(U,\calL)\lra \grG(U_B,\calL_B)$ is injective;
  \item moreover if $B$ is a field of characteristic $0$ then we have
    an isomorphism   $B\otimes_\R \grG(U,\calL)\isocan
    \grG(U_B,\calL_B)$.
  \end{enumerate}
\end{lem}

\begin{proof}
  The fact that $\grG(U,\calL)$ is finitely generated is the content
  of Theorem III.5.2 point a) in \cite{Hart}.

  We can refine the covering $U_i$ in such a way it has the same
  properties and moreover it is such that $\calL\ristretto_{U_i}$ is
  defined by a free $R_i$ module of rank $1$. We have the exact
  sequence:
  $$
  \begin{CD}
    0 \to \grG(U,\calL) @>{r_1}>> \prod \grG(U_i,\calL) @>{r_2}>>
    \prod \grG(U_i\cap U_j,\calL)
  \end{CD}
  $$
  where $r_1$ and $r_2$ are given by restriction of sections. In
  particular $\grG(U,\calL)$ is a submodule of $\prod \grG(U_i,\calL)$
  which is a free $\R$ module, hence, since $\R$ is a PID,
  $\grG(U,\calL)$ is a free $\R$ module.

  Let $ M$ denote the image of $r_2$. $M$ is a submodule of $\prod
  \grG(U_i\cap U_j,\calL)$, so also $M$ is a free $\R$ module. Write
  $r_2=\mi\circ r'$ with $r': \prod \grG(U_i,\calL)\to M$ and
  $\mi:M\to \prod \grG(U_i\cap U_j,\calL) $.  For any $\R$ algebra $B$
  we can tensor by $B$ and, since by definition
  $\grG(U_i\times_{\Spec(\R)}\Spec(B),\calL_B) = B \otimes_\R
  \grG(U_i,\calL)$, we get the exact sequence
  \begin{equation}\label{eq:M}
    \begin{CD}
      0 \to \grG(U,\calL)\otimes_\R B @>{r_1\otimes \id_B}>> \prod
      \grG(U_i\times_{\Spec(\R)}\Spec(B),\calL_B) @>{r' \otimes
        \id_B}>> M\otimes_\R B \to 0
    \end{CD}
  \end{equation}
  from which deduce that the map from $\grG(U,\calL)\otimes_\R B$ to
  $\grG(U_B,\calL_B)$ is injective. This proves $i)$ and $ii)$.

  Now assume that $B$ is a field of characteristic zero.  $\mi$ is an
  inclusion between free $\R$ modules, so since $B$ has characteristic
  zero we have that $\mi \otimes \id_B: M\otimes_\R B \lra
  \prod\grG(U_i\cap U_j,\calL)\otimes_\R B = \prod\grG((U_i\cap
  U_j)\times_{\Spec(\R)} \Spec(B) ,\calL_B)$ is injective. It follows
  that $M\otimes_\R B$ is the image of $r_2 \otimes \id_B$ and by the
  exact sequence \eqref{eq:M} we get that $\grG(U,\calL)\otimes_\R B$
  is equal to the space of sections of $\calL_B$ on
  $U_B$.
\end{proof}

Notice that by Remark \ref{oss:123}   3) we can apply this Lemma
to a toroidal compactification.  We obtain

\begin{lem}\label{lem:GIT}
  Let $Y$ be a smooth toroidal compactification of $G/H$ and let
  $\calL$ be a $K$ spherical line bundle on $Y$ generated by global
  sections. Then the restriction of sections from $Y$ to $Y_S$ induces
  an isomorphism $\grG(X,\calL)^{K} \isocan
  \grG(Y_S,\calL\ristretto_{Y_S})^{\tW}$.
\end{lem}

\begin{proof}
  We prove first that this map is injective.  Consider the pull back
  $\pi_{Y,K}^*(\calL)$. By hypothesis this is isomorphic to the trivial
  line bundle on $G/K$. So a trivialization of it induces inclusions
  $\Gamma(Y,\calL) \subset \mk[G/K]$ and
  $\Gamma(Y_S,\calL\ristretto_{Y_S}) \subset \mk[S_K]$. We get a
  commutative diagram
  $$
  \begin{CD}
    \Gamma(Y,\calL)^K @>>> \mk[G/K]^K \\
    @VVV @VVV \\
    \Gamma(Y_S,\calL\ristretto_{Y_S})^\tW @>>> \mk[S_K]^\tW
  \end{CD}
  $$ 
  where the horizontal maps are the induced by the pull back of
  sections and vertical maps are given by restriction of sections.
  Since the inclusions $G/H\subset Y$ and $S_H\subset Y_S$ are open the two
  horizontal maps are injective and by Theorem \ref{teo:Rich} also the
  right vertical map is injective. It follows that also the vertical
  map on the left is injective.

  In order to prove surjectivity it is enough to prove that we have
  enough invariants. First we prove this result in characteristic
  zero. Let $U$ be a $G$ submodule of $\mC[G/K]$ and $U_S$ be its image in
  $\mC[S_K]$. Observe that by Theorem \ref{teo:Rich} we have an
  isomorphism between $U^K$ and $U_S^\tW$.
  
  Set $U$ equal to the image of $\grG(Y,\calL)$ in $\mk[G/K]$. By
  Corollary \ref{cor:sezioniY} $ii)$, the restriction map
  $\grG(Y,\calL)\lra\grG(Y_S,\calL\ristretto_{Y_S})$ is surjective for
  any spherical line bundle generated by global sections. Thus $U_S$
  equals the image of $\grG(Y_S,\calL\ristretto_{Y_S})$ in $\mk[S_K]$
  and this implies our claim.

  Assume now that the base field $\mk$ is of arbitrary characteristic.
  The description of $\grG(Y_S,\calL\ristretto_{Y_S})$ as a $S$ module
  does not depend on the characteristic. It follows that there is a
  basis of $\grG(Y_S,\calL\ristretto_{Y_S})$ on which $\tW$ acts by
  permutations. Thus also the description of
  $\grG(Y_S,\calL\ristretto_{Y_S})^\tW$ and hence its dimension $d$
  does not depend on the characteristic. On the other hand by
  \ref{lem:piattezza}\, ii) we have that $d\leq \text{dim\ }
  \grG(Y_\mk,\calL_\mk)^K$.  Since $\grG(Y_\mk,\calL_\mk)^K$
  injects into $ \Gamma(Y_S,\calL\ristretto_{Y_S})^\tW$ everything
  follows.
\end{proof}

If we now set $$A_\calL:=\oplus_n\grG(Y,\calL^n)\hskip0.5cm
\text{and}\hskip0.5cm
B_\calL:=\oplus_n\grG(Y_S,\calL^n\ristretto_{Y_S})$$ we deduce from
the above Lemma that $\Proj(A^K_\calL)= \Proj(B^{\tW}_\calL)$ for any
spherical line bundle $\calL$ generated by global sections on a smooth
toroidal projective embedding $Y$ of $G/H$. In particular Theorem
\ref{teo:GIT} follows for such a compactification.

\subsection{Proof of Theorem \ref{teo:GIT}}\label{ssec:proofteoGIT}
We now prove Theorem \ref{teo:GIT} for any projective embedding $Y$ of
$G/H$. Consider an equivariant resolution $\tilde Y$ of the closure of
the image of $G/H$ in $Y\times X$. By construction this is a toroidal
compactification and we have a $G$ equivariant birational projective
morphism $\phi:\tilde Y \lra Y$. Clearly $\phi(\tilde Y_S) = Y_S$.

As already noticed at the beginning of section \ref{ssec:GITtoroidale}
we can assume $\calL$ to be $K$ spherical.  Let $\calM =\phi^*(\calL)$
and notice that this is a $K$ spherical line bundle on $\tilde Y$
generated by global sections. Notice also that since $Y$ is normal we
have $\grG(\tilde Y, \calM) = \grG(Y,\calL)$ and $A_\calM = A_\calL$.
We have the following commutative diagram
where the horizontal map are given by pull back of sections, and the
vertical maps are given by the restriction of sections:
$$
  \begin{CD}
    \Gamma(Y,\calL)^K @>{\simeq}>> \Gamma(\tilde Y,\calM)^K \\
    @VVV @VVV \\
    \Gamma(Y_S,\calL\ristretto_{Y_S})^\tW @>>> \Gamma(\tilde
    Y_S,\calM\ristretto_{\tilde Y_S})^\tW
  \end{CD}
$$ 
Now the vertical map on the right is an isomorphism by the result
obtained for a smooth toroidal compactification and the bottom map is
injective, since $\tilde Y_S \lra Y_S$ is surjective. So also the
vertical map on the left is an isomorphism. So $B_\calL^\tW \isocan
A_\calL^K$ and
$$
K\lGIT_\calL Y = \Proj(A^K_\calL) \isocan \Proj(B^{\tW}_\calL) = \tW\backslash Y_S
$$
as claimed. \hfill $\Box$

\section{Semistable points}
In this section we want to give a more geometric description of the
set of semistable points. We analyze first the case of flag varieties.

\subsection{Divisors of invariants in flag varieties}\label{ssec:ssGP}
We give first some definitions. If $I\subset \tDe$ we set $\Delta_I =
\Delta_0 \cup \{\gra \in \grD_1 \st \tal \in I\}$ and define $\Phi_I$
to be the subroot system of $\Phi$ spanned by $\Delta_I$. We let $P_I$
denote the corresponding parabolic subgroup of $G$ and $\grL_I =
\grL_{P_I}$ the set of characters of $P_I$.  We also set $\Pi_I =
\Pi\cap \Lambda_I$, $\Omega_{I,H} = \Omega_H \cap \Lambda_I$ and
$\Omega_I = \Omega_{I,\Hz}$.  We have
$$
\Omega_I = \bigoplus_{\tal\in \tDe \senza I}
\mZ\,\tom_\tal \qquad \Pi_I= \Omega_I + \bigoplus _{\gra \in \grD_e
  \senza \grD_I} \mZ\,\omega_\gra.
$$ 

Let us describe the set of invariant and semiinvariant sections with
respect to the action of $H$ of a line bundle $\calL_\lambda$ on
$G/P_I$.  Since if $\grl \in \grL_I^+$ we have that
$\grG(G/P_I,\calL_\grl) \isocan V_\grl^*$ (and is zero if $\grl$ is
not dominant). In this case we can apply directly the result of Vust
without introducing any further filtration.

If $\tal \in \tilde\Delta \senza I$ let $\bar p_\tal$ be a non zero
section of $\calL_{\tom_\tal}$ on $G/P_I$ invariant under the action
of $\Hz$.  Similarly if $\gra \in \grD_e \senza I$ let $\bar q_\gra$
be a non zero semiinvariant section of $\calL_{\om_\gra}$ on $G/P_I$.
If $\grl = \sum a_\tal \tom_\tal$ is dominant so that $a_\tal\geq 0$
for all $\tal$, we define $\bar p ^\grl = \prod_\tal \bar
p_\tal^{a_\tal}$ and similarly if $\grl = \sum_{\gra\in\grD_e} c_\gra
\om_\gra + \sum_{\tal\in \tDe_{ne}} c_\tal \tom_\tal$ is dominant in $
\Pi_I$ we define $\bar q^\grl=\prod_{\gra\in \grD_e} \bar
q_\gra^{c_\gra} \,\cdot \prod_{\tal \in \tDe_{ne}}\bar
p_\tal^{c_\tal}$.

We notice that up to a non zero scalar, $\bar p^\grl$ is the unique
$\Hz$ invariant section of $\calL_\grl$ and that it is $H$ invariant
if and only $\grl \in \Omega_H$.  Similarly $\bar q^\grl$ is the
unique $\Had$ semiinvariant section of $\calL_\grl$.

In particular if we set $R(G/P_I)=\bigoplus _{\grl \in \grL_I}
\Gamma(G/P_I,\calL_\grl)$, we have that the ring of invariants
$R(G/P_I)^\Hz$ is a polynomial ring in $\bar p_\tal$ for $\tal \notin
I$ and the ring of semiinvariants $R(G/P_I)^\Had_{si}$ is a polynomial
ring in $\bar p_\tal, \bar q_\beta$ for $\tal \in \tDe_{ne}\senza I$
and $\beta \in \Delta_e \senza \Delta_I$.

If $\grl=\sum c_\tal \tom_\tal \in \Omega$ we define the support of
$\grl$ as $\ssupp \grl = \{\tal\st c_\tal\neq 0 \}$. Also if $J\subset
\tDe \senza I$ we define $\bar p_J = \prod _{\tal \in J}\bar p_\tal$.

\begin{prp}\label{prp:ssGP}
  Let $J\subset \tDe \senza I$ then the equation $\bar p_J=0$ is
  reduced.

  Furthermore the divisor of the section $\bar p_{I^c}$ is the
  complement of the unique open $\Hz$ orbit in $G/P_I$, and we have
  $\Hz P_I = H P_I = \Had P_I$.
\end{prp}

\begin{proof}
  We start with the first assertion.  Let $D_J$ denote the divisor of
  $\bar p_J=0$ with reduced structure. Take $\grl \in \grL_I^+$ with
  the property that $\calL_\grl\isocan \calO(D_J)$ and let $\grf \in
  \grG(G/P_I,\calL_\grl)$ such that $\divi \grf = D_J$. We claim that
  $\grl=\tom_J:=\sum_{\tilde \alpha\in J}\tom_\tal$.
  
  Notice that by definition $\grf$ divides $\bar p_J$ and for big
  enough $n$ the section $\bar p_J$ divides $\grf^n$. So
  $\tom_J=\grl+\mu$ and $n\grl=\tom_J +\nu$ with $\mu$ and $\nu$
  dominant.  Moreover $D_J$ is an $\Had$ invariant so $\grf$ is an
  eigenvector under the action of $\Had$. In particular $\grl$ and
  also $\mu$ and $\nu$ are quasi spherical. Recall that $\tDe_{ne}
  =\{\tal \in \grD_1 : \gra $ is not exceptional$\}$. We can write
  $$
  \grl = \sum_{\tal \in \tDe_{ne}\senza I} c_\tal \tom_\tal
  +\sum_{\gra \in \grD_e : \tal \notin I}
  c_\gra \omega_\gra.
  $$
  Since $\grf$ divides $\bar p_J$ we obtain $c_\tal, c_\gra \leq 1$
  and $c_\tal= c_\gra =0$ for $\tal \notin J$. Alsosince $\bar p_J$
  divides $\grf^n$ we obtain $c_\tal, c_\gra \geq 1$ for $\tal \in J$.
  So $c_\gra=c_\tal=1$ if $\tal \in J$ and $\grl=\tom_J$ as claimed.

  Now let $U=\Hz\cdot [P_I]$ denote the open $\Hz$ orbit in $G/P_I$
  and $U'$ the complement of the divisor $D_{I^c}$.  Since $U'$ is
  $\Hz$ stable $U\subset U'$.
  
  We  claim that $U$ is affine. To see this is enough to prove that
  the Lie algebra of $\Hz\cap P_I $ is reductive.
  Indeed we have that this Lie algebra is equal to the Lie algebra of
  $L_I^\grs$ where $L_I$ is the Levi factor of $P_I$ containing $T$.
  
  Since $U$ is affine, by \cite{Hart2} Proposition 3.1 pg.~66, $D=G/P_I \senza U$ has pure 
  codimension one. Let
  $\calL_\grl \isocan \calO(D)$ and let $\grf \in
  \grG(G/P_I,\calL_\grl)$ be such that $\divi \grf =D$.
  
  Since $U\subset U'$ and $\bar p_{I^c}=0$ is reduced we have that
  $\bar p_{I^c}$ divides $\grf$. So $\calL_\grl$ is ample. Moreover
  the section $\grf$ must be an eigenvector under the action of $\Had$
  so $\grl$ is quasi spherical. This together with the fact that
  $\grl\in \grL_I$ easily implies that, up to a non zero constant,
  $$\grf= \prod_{\tal \in {I^c}\cap \tDe_{ne}} \bar p_\tal^{c_\tal} \quad\cdot
  \prod_{\gra \in \grD_e \st \tal \in {I^c}} \bar q_\gra^{c_\gra}$$
  with the exponents $c_\tal$ and $c_\gra$ positive.  On the other
  hand since $\grf$ is reduced we must have $c_\tal=c_\gra \leq 1$ for
  all $\tal, \gra$. So $\grl=\tom_{I^c}$ and $\grf=\bar p_{I^c}$ as
  claimed.

  Finally since the complement of the section $\bar p_J$  is
  stable by the action of $H$, or $\Had$ and  is a single
  $\Hz$ orbit, it is also a single $H$ or $\Had$ orbit.
\end{proof}

\subsection{Semistable points in a smooth toroidal compactification}
In this section we prove that the set of semistable points in a smooth
toroidal compactification does not depend on the choice of an ample
line bundle. 

We need to make few remarks on weights and convex functions.  We start
with a simple and well known Lemma on root systems.

\begin{lem}\label{lem:ssXradiciepesi}
  Let $\{\alpha_1,\ldots ,\alpha_r\}$ be a set of simple roots in a
  root system $R$, and $\{\omega_1,\dots ,\omega_r\}$ the
  corresponding set of fundamental weights. If $K\subset \{1,\ldots
  ,r\}$ every $\omega_j$ can be expressed as
  $$\omega_j=\sum_{h\in K}a_h\alpha_h+\sum_{k\notin K}b_k\omega_k$$
  with $a_h,b_k$  non negative rational numbers.
  
  Furthermore if $K=\Delta$, and $R$ is irreducible the $a_h$'s are
  strictly positive.
  \end{lem}

\begin{proof}
  If $r=1$ there is nothing to prove so we can proceed by induction.
  
  Assume $|K|<r$. The space $A$ and $B$ respectively spanned by the
  $\alpha_H$'s with $h\in K$ and by the $\omega_k$'s with $k\notin K$
  are mutually orthogonal. It follows that $\omega_j$ can be uniquely
  written as
 $$\omega_j=\gamma_j+\delta_j$$
 with $\gamma_j\in A$ and $\delta_j\in B$.
        
 If $j\notin K$ then $\gamma_j=0$ and there is nothing to prove.
  
 If $j\in K$, $\gamma_j$ is a fundamental weight for the root system
 in $A$ having the $\alpha_h$'s with  $h\in K$ as simple roots. Thus by
 induction $a_h\geq 0$ for each $h\in K$.  Write
 $\delta_j=\sum_{k\notin K}b_k\omega_k$. We get $b_k=\langle \delta_j,
 \gra_k\cech\rangle = - \langle \gamma_j, \gra_k\cech\rangle \geq 0$
 as desired.
  
 Assume now $|K|=r$.  Write $\omega_j=\sum_{h=1}^ra_{j,h}\alpha_h$ and
 notice that $0< ( \omega_j, \omega_j ) = a_{j,j} \,
 (\alpha_j,\omega_j)$. Since $(\alpha_j,\omega_j)>0$ we deduce that
 $a_{j,j} >0$. In particular we can write
  $$\alpha_j=\frac{\omega_j}{ a_{j,j}}+ \xi_j$$
  where $\xi_j$ is a linear combination of the elements in
  $\Delta\setminus \{\alpha_j\}$. It follows that
  $\omega_k=(a_{k,j}/a_{j,j})\omega_j+x$ with $x$ a linear combination
  of the elements in $\Delta\setminus \{\alpha_j\}$. Using the
  positivity of $a_{j,j}$ our claim now follows from the previous
  analysis applied in the case in which $K=\Delta\setminus
  \{\alpha_j\}$.

  It remains to show that in the irreducible case $a_{j,h}\neq 0$ i.e.
  $(\omega_j,\omega_h)\neq 0$ for each $j,h=1,\ldots r$. By
  contradiction assume that say $(\omega_1,\omega_2)= 0$. This means
  that $\omega_2$ lies in the span of $\alpha_2,\ldots \alpha_r$ and
  it is a fundamental weight for the root system having these roots as
  a set of simple roots. Let $R'$ be the irreducible component of this
  root system containing $\alpha_2$. We can assume that our ordering
  of simple roots is such that $\alpha_2,\ldots \alpha_s$ are a set of
  simple roots for $R'$.  By induction
  $\omega_2=\sum_{h=2}^sa_{2,h}\alpha_h$ with $a_{2,h}>0$. On the
  other hand
  $$0=(\omega_2,\alpha_1)=\sum_{h=2}^sa_{2,h}(\alpha_h,\alpha_1) $$ so
  that $(\alpha_h,\alpha_1)=0$ for each $h=1,\ldots ,s$. This clearly
  contradicts the irreducibility of $R$.
\end{proof}

Let $Y$ be a smooth toroidal compactification of $G/H$ and let $F_Y$
be the associated decomposition of the Weyl cochamber $C$.

We define the support $\ssupp \rho$ of a face $\rho$ of $F_Y$ as
$$\ssupp \rho:=\{\tal \in \tDe | \tal \text{\ is not identically zero  on\ } \rho\}.$$  

\begin{lem}\label{lem:torico1}
  Let $\bgrl= (\grl_\tau)_{\tau\in F_Y(\ell)} \in \SPic_0(Y)$ be such
  that $\calL_\bgrl$ is ample. Let $\rho$ be a
  face of $F_Y$. Then there exist a positive integer $n$ and $\mu \in
  \Omega^+_H$ such that $\ssupp \mu =\ssupp \rho$ and $\mu \in
  \calA(n\bgrl)$.
\end{lem}
\begin{proof} Let $S=\ssupp \rho $ and $T= \tDe\setminus \ssupp \rho$.
  Let $\tau$ be a maximal dimensional face containing $\rho$.  Since
  $\calL_\bgrl$ is ample, Corollary \ref{cor:sezioniY} $iv)$ implies
  that $\grl_\tau$ is regular dominant. So by Lemma
  \ref{lem:ssXradiciepesi} there exists a positive integer $n$ such
  that we can write $n\grl_\tau$ as
$$
n\grl_\tau = \sum_{\tal \in S} a_\tal \tom_\tal + \sum_{\tal \in T}
b_\tal \tal
$$
with   $a_\tal$ positive integers and $b_\tal$ non negative
integers.  Set $\mu = \sum_{\tal \in S} a_\tal \tom_\tal$.  We have
$\mu = n\grl_\tau$ on $\rho$ and $\mu \leq n \grl_\tau \leq n\bgrl$ on the
Weyl cochamber $C$ again by Corollary \ref{cor:sezioniY} $iv)$.
\end{proof}

The following Lemma is a sort of converse of Lemma \ref{lem:torico1}.

\begin{lem}\label{lem:torico2} 
  Let $\bgrl= (\grl_\tau)_{\tau\in F_Y(\ell)} \in \SPic_0$ be such
  that $\calL_\bgrl$ is ample. Let $\mu\in \Omega^+$ and $n$ a
  positive integer with $\mu\in \calA(n\bgrl)$ and $\mu = n\bgrl$ on
  $\rho$. Then $\ssupp \mu \supset \ssupp \rho $.
\end{lem}
\begin{proof}If $\rho$ is the zero face there is nothing to prove.
  Assume that $\rho$ has positive dimension. By eventually
  substituting $\bgrl$ with $n\bgrl$, let us also assume that $n=1$.

  Let $\rho(1)$ be the set of $1$ dimensional faces contained in
  $\rho$ and notice that $\ssupp \rho = \bigcup_{\vartheta \in
    \rho(1)} \ssupp \vartheta$.  So it is enough to prove the claim in
  the case of one dimensional faces.

Let $\rho$ be one dimensional  and  choose a non zero point
  $v $ in $\rho.$

  Take a face $\tau$ of maximal dimension containing $\rho$ and define
  $$
  \tau ^\rho = \{u \in \grL\cech_\mR \st v+t(v-u) \in \tau \text{ for
    some positive real number } t>0\}.
  $$

  Notice that $\mu \geq \grl_\tau$ on $\tau^\rho$. Indeed if $u \in
  \tau^\rho$ there is a positive $t$ such that $v+t(v-u) \in \tau$.
  Since $\mu \in \calA(\bgrl)$ we have $\mu(v+t(v-u))\leq \grl_\tau
  (v+t(v-u))$. But $\bgrl=\mu$ on $\rho$ so that
  $\mu(v)=\grl_\tau(v)$, so $\mu (u)\geq \grl_\tau(u)$.

  Since the support of $F_Y$ equals $C$ it is then clear that
$$
\bigcup_{\tau\in F_Y(\ell) \st \tau \supset \rho} \tau^\rho = \{u \in
\grL_\mR\cech \st \langle \tal, u\rangle \leq 0 \mforall \tal \notin
\ssupp \rho\}.
$$
Thus every $\tal \in \ssupp \rho$ lies in at least one of the sets $
\tau^\rho$. It follows that $$ \langle \mu , \tal \rangle \geq \langle
\grl_\tau , \tal \rangle >0
$$
since $\bgrl$ is ample. Thus $\tal \in\ssupp \mu$.\end{proof}

We can now characterize the set of semistable points with respect to
an ample line bundle $\calL$ on $Y$.

Consider the $G$ equivariant projection $\phi:Y\lra X$ onto the
wonderful compactification $X$. For each $\mu \in \Omega$ we pull back
the $\Hz$ invariant section $p^\mu$. This is an $\Hz$ invariant
section of $\grG(Y,\phi^*(\calL_\mu))$ and we denote it by the same
symbol $p^\mu$. For any subset $I\subset \tilde\Delta$ we set $p_I :=
\prod _{\tal \in I}p^{\tilde\omega_{\tilde\alpha}}$.  We also remark
that the condition of being semistable does not depend on the group
$K$ between $\Hz$ and $\Had$.  Indeed, by the description of invariant
sections, if $f$ is a $\Hz$ invariant section which does not vanish on
$x$ there exists an integer $n$ such that $f^n$ is invariant under
$\Had$.  In view of this we will speak of semistable points without
specifying the group $K$.

\begin{prp}
  \label{prp:Yss}Let $Y$ be a smooth projective toroidal embedding of
  $G/H$ and let $\calL$ be an ample line bundle on $Y$.  Let $\rho$ be
  a face of $F_Y$ and let $O_\rho$ be the corresponding $G$ orbit. A
  point $x\in O_\rho$ is $\calL$ semistable if and only if $p_{\ssupp
    \rho} (x)\neq 0$.
\end{prp}
\begin{proof}
  Let $\calL=\calL_\bgrl$.  Let $\rho$ be a face of $F_Y$ and $x\in
  O_\rho$ . Assume that $p_{\ssupp \rho}(x)\neq 0$. By Lemma
  \ref{lem:torico1} there exist a positive integer $n$ and a dominant
  weight $\mu$ such that $\mu \in \calA(n\bgrl)$, $\mu = n\bgrl$ on
  $\rho$ and $\ssupp \mu = \ssupp \rho$. Since $p^\mu$ is a product of
  the sections $p_{\tal}$ for $\tal \in \ssupp \rho$ we have that
  $p^\mu(x)\neq 0$.  Thus $s^{n\bgrl - \bmu}p^\mu$ is an $H$ invariant
  section of $\calL_\bgrl^n$ not vanishing on $x$ and $x$ is
  semistable.

  Conversely let $x$ be semistable. Then by the description of
  invariants given in the proof of Theorem \ref{glinvarianti} there
  exists a positive integer $n$ and a dominant spherical weight
  $\mu\in \calA(n\bgrl)$ such that $s^{n\bgrl-\bmu}p^\mu(x)\neq 0$.
  The condition $s^{n\bgrl-\bmu}(x)\neq 0$ implies $\mu = n\bgrl$ on
  $\rho$ so we can apply Lemma \ref{lem:torico2} and we deduce that
  $\ssupp \mu \supset \ssupp \rho$. In particular $p_\tal(x) \neq 0$
  for all $\tal \in \ssupp \rho$ or equivalently $p_{\ssupp
    \rho}(x)\neq 0$.
\end{proof}

 The above Proposition has the following
 
\begin{cor}\label{indipendence} Let $\calL$ and $\calL'$ be two ample line bundles on $Y$. Then
$Y^{ss}(\calL)= Y^{ss}(\calL')$.
\end{cor}
In view of this Corollary we shall from now on say that a point is
semistable if it is semistable with respect to any ample line bundle
and we shall denote the set of semistable points by $Y^{ss}$.
\medskip

 We now give a more set theoretic description of semistable points.
 Take a face $\rho$ of the fan $F_Y$ and denote by $O_\rho$ the
 corresponding $G$ orbit in $Y$.  Set $I = \ssupp \rho$ and consider
 the $G$ equivariant projection $\phi:Y\lra X$. Remark that for any
 $\eta$ in the relative interior of $\rho$ the point $y_\eta$ depends
 only on $\ssupp \rho$ and not on the choice of $\eta$. Thus we can
 denote this point by $y_\rho$.  By definition for such one parameter
 subgroup $\eta$ we have $\phi(y_\eta)=x_\eta$. In particular it
 follows that $\phi(O_\rho)= O_I$ the open orbit in $X_I$ and that the
 projection $\phi:O_\rho\to O_I$ is a $G$ equivariant fibration.

By \cite{DP1} we have a $G$ equivariant projection 
$\pi_I:O_I  \to G/P_\Ic$ with $\Ic=\tDe\setminus I$. Composing we get a fibration
$$
\gamma_I:=\pi_I\circ
\phi:
  O_\rho \to G/P_\Ic
$$
whose fiber over the point $\pi_I\circ\phi(y_\rho)$ we denote by $F_\rho$.

In view of Proposition  \ref{prp:ssGP} and Proposition \ref{prp:Yss} we immediately get

\begin{prp}\label{prp:OF}
  A point $x$ in $O_\rho$ is semistable if and only if its $H$ orbit
  intersects $F_\rho$. So we have that $O_\rho^{ss}:=Y^{ss}\cap O_\rho$ is equal to  $ \Hz F_\rho$ (and also to $ K  F_\rho$ and
 $\Had F_\rho$).
\end{prp}
\begin{proof} This is clear since by Proposition \ref{prp:ssGP} the
  section $p_{\ssupp \rho} $ does not vanish exactly on the inverse
  image under $\gamma_I$ of the open $\Hz$ orbit in
  $G/P_\Ic$.\end{proof}

By Proposition \ref{prp:OF} we then have that $O_\rho^{ss}
=H\times_{H\cap P_\Ic}F_\rho$.

Let us now take a subset $J\subset\tDe$.  We denote by $L$ the
standard Levi factor of $P_J$ and recall that $L$ is $\sigma$ stable
and that if $U_{P_J}$ is the unipotent radical of $P_J$,
$\sigma(U_{P_J})=(U_{P_J})^-$ the opposite unipotent.

\begin{lem}\label{lem:UUL} For any subset $J\subset\tDe$ we have
    $$K \cap L =K \cap P_J.$$
\end{lem}
\begin{proof}
  It is enough to analyze the case of $K=\Had$. Since our problem is a
  problem of support we can work with the associated reduced subgroup
  $\Had_{red} = \{x \in G \st x\grs(x)^{-1} \in Z(G)\}$.
  
  Take $x = m \, u \in P_J \cap \Had_{red}$ with $m\in L$ and $u \in
  U_{P_J}$. Then clearly $u\grs(u)^{-1} \in L$. It follows that
  $\grs(u)^{-1}\in U^{-}_{P_J}\cap P_J=\{1\}$ thus $u=1$ and $x\in L$
  as desired.
\end{proof}
 
 Going back to the semistable points in the orbit $O_\rho\subset Y$, we    get

\begin{prp}\label{lariduzione} Let $L$ be the standard Levi factor of $P_\Ic$. Set $K_L=K\cap L$. We have a $K$ equivariant isomorphism
  $$O_\rho^{ss}\simeq K\times_{K_L}F_{\rho}.$$ 
  In particular this induces a closure preserving bijection between $K$ orbits in $O_\rho^{ss}$ and $K_L$ orbits in $F_\rho$.
\end{prp}

Now consider the fiber $F_I$ of  $\pi_I$  containing $\phi(y_\rho)$. We know from \cite{DP1} that the solvable radical of $P_{I^c}$ acts trivially on $F_I$ and that 
$F_I=\overline L/\overline H_{\overline L}$ where $\overline L$ is the adjoint quotient of $L$ and $\overline H_{\overline L}$ is the subgroup fixed by the involution induced by $\sigma$. 

By the description of $Y$ given in Theorem \ref{teo:toroidali} it now follows that if we set $L_{\rho}$ equal to the quotient of $L$ modulo the subgroup in the center of $L$ generated by the one parameter subgroups $\eta$ with $\eta\in \rho$, $F_\rho$ can be identified with $L_\rho/H_\rho$, where $H_\rho$ is isogenous to the subgroup fixed by the involution on $L_\rho$ induced by $\sigma$. 

Under this identification $S_{H_\rho}$ coincides with $Y_S\cap O_\rho$. We thus can apply Remark \ref{ilridutt}  and  deduce
\begin{prp}\label{lechiusse}
 Let $x\in Y_S\cap O_\rho$. Then the orbit $Kx$ is closed in $O^{ss}_\rho$.
\end{prp}

\section{Closed semistable orbits}
In this section we prove that in the case of a smooth toroidal
compactification $Y$ of $G/H$ the orbits through the elements of $Y_S$
are the closed semistable orbits. In this way we give a geometric
counterpart to Theorem \ref{teo:GIT}.

We show first that the closure of $K$ orbits does not interacts with
$G$ orbits.

\begin{prp}\label{lem:Gorbiteechiusure}
  Let $Y$ be a smooth toroidal compactification of $G/H$, let $\calL$
  be an ample line bundle on $Y$. Let $x,y \in Y^{ss}(\calL)$. If $y
  \in \overline K x$ then $y \in G\cdot x$.
\end{prp}
\begin{proof}
  As we have pointed out in Corollary \ref{indipendence}, the set of
  semistable points of $Y$ does not depend on the choice of the ample
  line bundle $\calL$.
   
  Let $\calL=\calL_\bgrl$. For $n$ big enough we have that
  $\calL_{n\bgrl+\bgra_D}$ is also ample for every $D \in \grD_Y$. In
  particular for a large enough positive integer $m$, we can find
  invariant sections $f_D \in \grG(Y,\calL_{m(n\bgrl+\bgra_D)})^K$ and
  $f \in \grG(Y,\calL_{mn\bgrl})^K$ such that $f_D(y)\neq 0$ and $f(y)
  \neq 0$.
  
  Set now $$U =\{z \in Y \msuchthat f(z)\neq 0 \mand f_D(z) \neq 0
  \mforall D \in \grD_Y\}.$$ The set $U$ is an open affine $H$
  invariant subset of $Y^{ss}$ with the property that if
  $$\pi:Y^{ss}\to K\lGIT Y$$
  is the quotient morphism $U=\pi^{-1}(\pi(U)).$
  In particular $x\in U$.
   Furthermore on $U$ the line bundles
  $\calL_{m(n\bgrl+\bgra_D)}$ and $\calL_{mn\bgrl}$ have a $H$
  equivariant trivialization. 
  
  It follows that also the line bundle
  $\calL_{m\gra_D}=\calL_{m(n\bgrl+\bgra_D)}\otimes
  \calL_{mn\bgrl}^{-1}$ has an $H$ equivariant trivialization $\phi_D$
  on $U$.  Thus we can consider the $H$ invariant function $t_D =
  \phi_D(s_D^m)$ on $U$. Now by Theorem \ref{teo:toroidali} a $G$
  orbit in $Y$ is determined by the set of $D$ such that $s_D$
  vanishes on the orbit.  This implies our claim.
\end{proof}  

\begin{oss}\label{oss:closed1}
  The following simple example shows that Proposition
  \ref{lem:Gorbiteechiusure} does not hold for a non toroidal $Y$.
  Take the compactification $\mP(\text{End}(\mC^3))$ of $PSL(3)$, so
  $G=SL(3) \times SL(3)$ and $K$ is the normalizer of the diagonal
  copy of $SL(3)$. The elements
$$
x=\left(\begin{matrix}
  0 & 1 & 0 \\ 0 & 0 & 0 \\ 0 & 0 & 1
\end{matrix}\right)
\quad \mand \quad
y = \left(\begin{matrix}
  0 & 0 & 0 \\ 0 & 0 & 0 \\ 0 & 0 & 1
\end{matrix}\right)
$$
give a counterexample to the statement in the Proposition.
\end{oss}

Let $U$ be any $G$ stable open subset of smooth toroidal projective
embedding $Y$. The proof of Proposition \ref{lem:Gorbiteechiusure} 
implies
  
\begin{prp} $\pi^{-1}(\pi(U\cap Y^{ss}))=U\cap Y^{ss}$. Furthermore
  $$\pi_{ |U\cap Y^{ss}}:U\cap Y^{ss}\to \tilde W\backslash U_S$$
  is a well defined quotient map. 
\end{prp}

Notice however that, as the following example shows, there are ample
line bundles $\calL$ on $Y$ such that if we restrict $\calL$ to $U$,
$U^{ss}(\calL)\neq U\cap Y^{ss}$. Indeed, take $Y$ equal to the
wonderful embedding of $PSL(3)$, $\calL=\calL_{\tom_1+\tom_2}$ and $U$
equal to the complement of the divisor of equation $s_{\tal_1}=0$.
Then $U$ is isomorphic to the open set in $\mP(\text{End}(\mC^3))$ of
classes of matrices of rank at least 2 and there is an invariant
namely $s_{\tilde \alpha_1}^{-1}p_{\tal_1}^3$, which up to a constant
gives the cube of the trace, defined on $U$ and not vanishing on
$$x=
\left(\begin{matrix} 0 & 1 & 0 \\ 0 & 0 & 0 \\ 0 & 0 & 1
\end{matrix}\right)
$$
while $x$ is not semistable in $Y$.

\medskip

From Propositions \ref{lechiusse} and \ref{lem:Gorbiteechiusure} we
finally get
\begin{teo}\label{teo:closed}Let $Y$ be a smooth toroidal
  compactification of $G/H$ and let $\Hz \subset K \subset \Had$.  Let
  $x \in Y^{ss}$ then $K x$ is closed in $Y^{ss}$ if and only if $K x
  \cap Y_S \neq \vuoto$.
\end{teo}
\begin{proof} From the proof of Theorem \ref{teo:GIT} we get that
  $Y_S\subset Y^{ss}$ and the fact that if $x\in Y_S$ its orbit is
  closed is Proposition \ref{lechiusse}.

  On the other hand since $K\backslash\backslash Y=\tilde W\backslash
  Y_S$, the restriction of the quotient map
$$\pi :Y^{ss}\to K\backslash\backslash Y$$
to $Y_S$ is surjective. Given $p\in K\backslash\backslash Y$,
$\pi^{-1}(p)$ contains a unique closed orbit and this by the first
part is necessarily the orbit of a point in $\pi^{-1}(p)\cap
Y_S$\end{proof}

\bibliographystyle{plain}

\begin{thebibliography}{10}

\bibitem{Bifet}
E.~Bifet.
\newblock On complete symmetric varieties.
\newblock {\em Adv. Math.}, 80(2):225--249, 1990.

\bibitem{Brion}
M.~Brion.
\newblock Groupe de {P}icard et nombres caract\'eristiques des vari\'et\'es
  sph\'eriques.
\newblock {\em Duke Math. J.}, 58(2):397--424, 1989.

\bibitem{CM1}
R.~Chiriv{\`{\i}} and A.~Maffei.
\newblock The ring of sections of a complete symmetric variety.
\newblock {\em J. Algebra}, 261(2):310--326, 2003.

\bibitem{DP1}
C.~De~Concini and C.~Procesi.
\newblock Complete symmetric varieties.
\newblock In {\em Invariant theory (Montecatini, 1982)}, pages 1--44. Springer,
  Berlin, 1983.

\bibitem{DP2}
C.~De~Concini and C.~Procesi.
\newblock Complete symmetric varieties. {I}{I}. {I}ntersection theory.
\newblock In {\em Algebraic groups and related topics (Kyoto/Nagoya, 1983)},
  pages 481--513. North-Holland, Amsterdam, 1985.

\bibitem{DS}
C.~De~Concini and T.~A. Springer.
\newblock Compactification of symmetric varieties.
\newblock {\em Transformation Groups}, 4(2-3):273--300, 1999.

\bibitem{Hart2}
R.~Hartshorne.
\newblock {\em Ample subvarieties of algebraic varieties}.
\newblock {Notes written in collaboration with C.~Musili}.
\newblock Springer-Verlag, Berlin, 1970.
\newblock Lecture Notes in Mathematics, No. 156.

\bibitem{Hart}
R.~Hartshorne.
\newblock {\em Algebraic geometry}.
\newblock Springer-Verlag, New York, 1977.
\newblock Graduate Texts in Mathematics, No. 52.

\bibitem{Helg}
S.~Helgason.
\newblock A duality for symmetric spaces with applications to group
  representations.
\newblock {\em Advances in Math.}, 5:1--154, 1970.


\bibitem{Jantzen}
J.~C. Jantzen.
\newblock {\em Representations of algebraic groups}, volume 131 of {\em Pure
  and Applied Mathematics}.
\newblock Academic Press Inc., Boston, MA, 1987.

\bibitem{Knapp}
A.~W. Knapp.
\newblock {\em Lie groups beyond an introduction}, volume 140 of {\em Progress
  in Mathematics}.
\newblock Birkh\"auser Boston Inc., Boston, MA, second edition, 2002.

\bibitem{KoRa}
B.~Kostant and S.~Rallis.
\newblock Orbits and representations associated with symmetric spaces.
\newblock {\em Amer. J. Math.}, 93:753--809, 1971.

\bibitem{Newstead}
P.~E. Newstead.
\newblock {\em Introduction to moduli problems and orbit spaces}, volume~51 of
  {\em Tata Institute of Fundamental Research Lectures on Mathematics and
  Physics}.
\newblock Tata Institute of Fundamental Research, Bombay, 1978.

\bibitem{Rich}
R.~Richardson.
\newblock Orbits, invariants, and representations associated to involutions of
  reductive groups.
\newblock {\em Invent. Math.}, 66(2):287--312, 1982.

\bibitem{Ruzzi}
A.~Ruzzi.
\newblock Complete toroidal symmetric varieties; projective normality; toric
  varieties.
\newblock {\em J. of Algebra}, 318(1):302--322, 2007.

\bibitem{Springer}
T.~A. Springer.
\newblock {\em Linear algebraic groups}, volume~9 of {\em Progress in
  Mathematics}.
\newblock Birkh\"auser Boston Inc., Boston, MA, second edition, 1998.

\bibitem{Spr1}
T.~A. Springer.
\newblock The classification of involutions of simple algebraic groups.
\newblock {\em J. Fac. Sci. Univ. Tokyo Sect. IA Math.}, 34(3):655--670, 1987.

\bibitem{Strick}
E.~Strickland.
\newblock A vanishing theorem for group compactifications.
\newblock {\em Math. Ann.}, 277(1):165--171, 1987.

\bibitem{Vust}
T.~Vust.
\newblock Op\'eration de groupes r\'eductifs dans un type de c\^ones presque
  homog\`enes.
\newblock {\em Bull. Soc. Math. France}, 102:317--333, 1974.

\end{thebibliography}
\def\dbar{\leavevmode\hbox to 0pt{\hskip.2ex \accent"16\hss}d}

\end{document}